\documentclass{daj}

\usepackage{amsmath}
\usepackage{amsfonts}
\usepackage{amsthm}

\theoremstyle{plain}
\newtheorem{theorem}{Theorem}[section]
\newtheorem{lemma}[theorem]{Lemma}
\newtheorem{corollary}[theorem]{Corollary}

\newtheorem{proposition}[theorem]{Proposition}

\newtheorem*{theorem*}{Theorem}

\newtheorem*{problem*}{Problem}

\theoremstyle{definition}
\newtheorem{definition}[theorem]{Definition}

\newtheorem*{claim*}{Claim}
\newtheorem*{remark}{Remark}
\newtheorem*{remarks}{Remarks}

\renewcommand{\P}{{\mathbb P}}
\newcommand{\e}{\mathrm{e}}
\newcommand{\tM}{{\widetilde M}}
\newcommand{\R}{{\mathbb R}}
\newcommand{\D}{{\mathbb D}}
\newcommand{\T}{{\mathbb T}}
\newcommand{\st}{{\text{\rm st}}}
\newcommand{\un}{{\text{\rm un}}}
\newcommand{\er}{{\text{\rm er}}}
\newcommand{\one}{\mathbf{1}}
\newcommand{\norm}[1]{\left\Vert #1\right\Vert}
\newcommand{\nnorm}[1]{\lvert\!|\!| #1|\!|\!\rvert}
\newcommand{\lip}{{\text{\rm Lip}}}

\dajAUTHORdetails{%
  title = {Ergodicity of the Liouville System Implies the Chowla Conjecture}, 
  author = {Nikos Frantzikinakis},
  plaintextauthor = {Nikos Frantzikinakis},
  plaintexttitle = {Ergodicity of the Liouville System Implies the Chowla Conjecture}, 
  keywords ={Multiplicative functions, Liouville function, M\"obius function,   Chowla conjecture, Elliott conjecture,  Sarnak conjecture, Gowers
  uniformity, inverse theorems.},
}   

\dajEDITORdetails{%
	year={2017},
	number={19},
	received={18 September 2017},   
	published={12 December 2017},  
	doi={10.19086/da.2733},       
}   

\begin{document}

\begin{frontmatter}[classification=text]

\title{Ergodicity of the Liouville System Implies the Chowla Conjecture} 


\author[Nikos]{Nikos Frantzikinakis}

\begin{abstract}
The Chowla conjecture asserts that the values of the Liouville function form a normal sequence of  plus and minus ones. Reinterpreted in the language of ergodic theory it asserts that the Liouville function is generic for the Bernoulli measure on the space of  sequences with values plus or minus one. We show that these statements are implied by the much weaker hypothesis that the Liouville function is generic for an ergodic measure. We also give variants of this result related to a conjecture of Elliott on correlations of  multiplicative functions with values on the unit circle.  Our argument has an ergodic flavor and combines recent results in analytic number theory,  finitistic and infinitary decomposition results involving uniformity seminorms, and qualitative equidistribution results on nilmanifolds.
\end{abstract}
\end{frontmatter}


 \section{Introduction and main results}
 \subsection{Introduction}
 Let $\lambda\colon \mathbb{N}\to \{-1,1\}$ be  the Liouville function which  is defined to be $1$ on integers with an even number of prime factors, counted with  multiplicity, and $-1$ elsewhere. It is generally believed
 that  the values  of the Liouville function enjoy various  randomness properties and one manifestation of this principle
 is an old conjecture of Chowla
 \cite{Ch65} which  asserts that for all $\ell\in \mathbb{N}$ and all  distinct  $n_1,\ldots, n_\ell\in \mathbb{N}$ we have
 \begin{equation}\label{E:Chowla}
 	\lim_{M\to \infty} \frac{1}{M}\sum_{m=1}^M \lambda(m+n_1)\cdots \lambda(m+n_\ell)=0.
 \end{equation}
 The conjecture is known  to be true only for $\ell=1$;  this case is elementarily equivalent to the prime number theorem.
 For $\ell=2$ and for all odd values of $\ell\in \mathbb{N}$,  a  variant involving logarithmic averages
 was recently established by Tao  \cite{Tao15} and   Tao, Ter\"av\"ainen \cite{TT17} respectively, and  an averaged form of the Chowla conjecture
 was established by
 Matom\"aki, Radziwi{\l}{\l},  and Tao
 \cite{MRT15} using a recent breakthrough of Matom\"aki and Radziwi{\l}{\l} \cite{MR15} concerning  averages of bounded multiplicative functions on typical short intervals. For all $\ell \geq 2$ the conjecture remains open for Ces\`aro averages and for all even $\ell\geq 4$ it  remains open for logarithmic averages.
 It is a consequence of the previous results   that all size three sign patterns are taken by  consecutive values of $\lambda$  with positive lower density  \cite{MRT15b}  (and in fact with logarithmic density $1/8$ \cite{TT17})
 and  all size four sign patterns are taken  with positive lower density \cite{TT17}. Similar results are not known for
 patterns of longer size  and in fact out of the $2^\ell$ possible size $\ell$ sign patterns
 only $\ell+5$ of them are known to be taken by consecutive values of $\lambda$  with positive lower density  \cite{MRT15b} (the Chowla conjecture predicts that all $2^\ell$ patterns are taken and each one  with density $2^{-\ell}$).

 We can   reinterpret the Chowla conjecture  in the language of ergodic theory, hoping that  this offers some appreciable advantage (a point of view also taken for example in \cite{AKLR14, Sa10}). Assuming for the moment that the limit on the left hand side of \eqref{E:Chowla} exists for all $\ell\in \mathbb{N}$ and  $ n_1,\ldots,n_\ell\in \mathbb{N}$,
 we introduce in a natural way
 a dynamical  system (see Proposition~\ref{P:correspondence}),  which we call  the ``Liouville system''. The Chowla conjecture  implies that  this  system is a Bernoulli system, but
 up to now, randomness properties of the Liouville system that are much weaker than independence remain elusive. For instance, it is not known whether this system is of positive entropy,  weakly mixing,  or even ergodic. We prove that the much weaker hypothesis of ergodicity implies that the Liouville system is Bernoulli and implies the Chowla conjecture. This can be stated informally  as follows (see Section~\ref{SS:Chowla} for the formal statements of our main results):
 \begin{theorem*}
 	If the Liouville system is ergodic, then the Chowla conjecture is satisfied.
 \end{theorem*}
 Thinking of $\lambda$ as a point on the sequence space $\{-1,1\}^\mathbb{N}$, we can  reformulate this result using notation from \cite[Definition~3.4]{Fu81a} as follows:  If the Liouville function is generic for an ergodic measure on the sequence space, then the Chowla conjecture is satisfied.

 An implicit assumption made in these statements is
 that the Liouville function admits correlations for Ces\`aro averages on the integers.
 In Section~\ref{SS:Chowla} we   give results
 that do not depend on such strong hypothesis; we work with sequences of intervals with left end equal to $1$ along which the Liouville function admits correlations for logarithmic averages (such sequences are guaranteed  to exist),  and  Theorem~\ref{T:ErgodicChowla}  states that ergodicity of the corresponding measure preserving system implies that the Chowla conjecture holds for logarithmic averages on the same sequence of intervals. We deduce in Corollary~\ref{C:ErgodicLogChowla}  ergodicity conditions which imply that the Chowla conjecture holds  for logarithmic averages taken on the full sequence of integers.

 Three main ingredients enter the proof of Theorem~\ref{T:ErgodicChowla}:
 \begin{enumerate}
 	\item A recent result of Tao  (see Theorem~\ref{T:Tao}) enables to reduce the Chowla conjecture for logarithmic averages to  a  local uniformity property of the Liouville function (this is the only reason why some of our statements involve logarithmic averages). Our goal then becomes to prove this
 	uniformity property (stated in Theorem~\ref{T:uniformity}).
 	
 	\item An inverse theorem for local  uniformity seminorms, which takes
 	a particularly useful form for ergodic sequences (see Theorem~\ref{T:ergodicInverse}). In order to prove it we use  both infinitary
 	and finitary decomposition results (see Propositions~\ref{P:InfiniteDec} and \ref{P:FiniteDec}).
 	The former is proved via
 	an ergodic inverse theorem of Host and Kra \cite{HK05a},
 	and the latter  via
 	a  finitistic  inverse theorem of Green, Tao, and Ziegler \cite{GTZ12c}.  The ergodicity of the sequence is essential; without this assumption we are led to  conditions that we are unable to verify for the Liouville function.

 	\item An asymptotic orthogonality property of the Liouville function with nilsequences taken on  typical short intervals (see Proposition~\ref{P:Discorrelation'}); this is needed in order to  verify that the aforementioned inverse theorem is applicable to the Liouville function. For Abelian nilsequences the orthogonality  property
 	follows from recent work of
 	Matom\"aki, Radziwi{\l}{\l},  and Tao
 	(see Proposition~\ref{P:stronaperiodic}). For general nilsequences additional tools are needed;
 	the heart of the argument is a result of purely dynamical context
 	(see Proposition~\ref{P:orthogonality}) and the only extra  number-theoretic input needed is the orthogonality criterion of Lemma~\ref{L:katai}.
 \end{enumerate}
 Our argument also works for the M\"obius function;
 hence, ergodicity of the M\"obius function  implies  a related Chowla-type result, and as a consequence, it also implies a conjecture of Sarnak~\cite{Sa10,Sa12}
 stating that the M\"obius function is uncorrelated with any bounded deterministic sequence.
 Moreover,  our argument shows that
 every ergodic strongly aperiodic
 multiplicative function (see Definition~\ref{D:uniformly})
 is locally uniform (see Theorem~\ref{T:uniformity}). This last property implies  an Elliott-type result for this class of  multiplicative functions (see Theorem~\ref{T:ErgodicElliott}) which  in turn implies  non-correlation with any bounded deterministic sequence.

\subsection{Main results} In this subsection we give the precise statements of our main results, modulo notation that appears in the next section.
We let $[N]=\{1,\ldots,N\}$.
\subsubsection{Ergodicity and Chowla's conjecture}\label{SS:Chowla}
Our main result is the following  (the notions used are explained in Section~\ref{S:notation}):
\begin{theorem}\label{T:ErgodicChowla}
	Let $N_k\to\infty$ be a sequence of integers and $\mathbf{I}=([N_k])_{k\in\mathbb{N}}$.
	If the Liouville or the M\"obius function is ergodic  for logarithmic averages on $\mathbf{I}$,
	then it satisfies the  Chowla conjecture for logarithmic averages on $\mathbf{I}$.
\end{theorem}
\begin{remarks}
	$\bullet$ Since for every $\ell\in\mathbb{N}$ each size $\ell$ sign pattern is expected to be taken by consecutive values of $\lambda$,
	we cannot substitute the intervals $[N_k]$ with arbitrary  intervals that do not start at $1$ and have lengths increasing to infinity. The same comment applies to the results of the next subsection.

	$\bullet$ We stress that if we assume ergodicity of the Liouville system
	for Ces\`aro (instead of logarithmic) averages on ${\bf I}$, our argument does not
	allow to deduce that the Chowla conjecture is satisfied for Ces\`aro  averages on ${\bf I}$.
\end{remarks}

Since  for every $a\in \ell^\infty(\mathbb{N})$ and  $\mathbf{I}=([N_k])_{k\in\mathbb{N}}$, $N_k\to\infty$,
there exists a subsequence ${\bf I}'$ of ${\bf I}$ on which the sequence $a$ admits correlations,
we deduce from Theorem~\ref{T:ErgodicChowla} the following:
\begin{corollary}\label{C:ErgodicLogChowla}
	Suppose that whenever the Liouville (or the M\"obius) function admits correlations for logarithmic averages on a sequence of intervals $\mathbf{I}=([N_k])_{k\in\mathbb{N}}$, $N_k\to\infty$,
	the induced measure preserving system is ergodic.
	Then the Liouville (resp. the M\"obius) function  satisfies the
	Chowla conjecture for logarithmic averages on $([N])_{N\in \mathbb{N}}$.\footnote{It  then follows from  a recent result in \cite{GKL17} that there exists $M_k\to\infty$ such that $\lim_{k\to \infty} \mathbb{E}_{m\in [M_k]} \lambda(m+n_1)\cdots \lambda(m+n_\ell)=0$ for all $\ell\in \mathbb{N}$ and distinct $n_1,\ldots,n_\ell\in\mathbb{N}$. }
\end{corollary}
Note that this result does not  impose an assumption of existence of correlations for logarithmic averages of the Liouville function; it rather states that
in order to  verify the Chowla conjecture for logarithmic averages
it suffices to  work under the assumption that correlations exist on some sequence of intervals
 $\mathbf{I}=([N_k])_{k\in\mathbb{N}}$, $N_k\to\infty$,
and then for any such sequence  $\mathbf{I}$ verify that the induced measure preserving system is ergodic.

Since convergence of Ces\`aro averages on $\mathbf{I}=([N])_{N\in \mathbb{N}}$ implies convergence to the same limit of logarithmic averages on $\mathbf{I}$, we deduce the result stated in the introduction:
\begin{corollary}\label{C:ErgodicChowla}
	If the Liouville or the M\"obius function is ergodic for Ces\`aro averages on  $\mathbf{I}=([N])_{N\in \mathbb{N}}$, then it satisfies the Chowla conjecture
	for Ces\`aro averages on  $\mathbf{I}$.
\end{corollary}
Further analysis of structural properties of measure preserving systems naturally associated with the Liouville or the M\"obius function appear in the recent article of the author and Host \cite{FH17}. The direction taken in  \cite{FH17} is complementary to the
one in this article and the techniques used very different.

\subsubsection{Ergodicity and Elliott's conjecture}\label{SS:Elliott}
We give a variant of our main result which applies to correlations of arbitrary multiplicative functions with values on the unit circle. This relates to logarithmically averaged variants of  conjectures made by Elliott in  \cite{El90,El94}.
\begin{theorem}\label{T:ErgodicElliott}
	Let $f_1\in \mathcal{M}$ be a strongly aperiodic multiplicative function which is  ergodic
	for  logarithmic  averages on $\mathbf{I}=([N_k])_{k\in\mathbb{N}}$, $N_k\to\infty$. Then for every $s\geq 2$ we have that
	\begin{equation}\label{E:Elliott}
		\mathbb{E}^{\log}_{m\in \mathbf{I}}\,  f_1(m+n_1)\cdots f_s(m+n_s)=0
	\end{equation}
	holds for  all  $f_2,\ldots, f_s\in \mathcal{M}$ and all distinct $n_1,\ldots, n_s\in \mathbb{N}$.
\end{theorem}
Elliott conjectured that the conclusion holds for  Ces\`aro averages without the ergodicity assumption and  under the weaker assumption of aperiodicity (which coincides with strong aperiodicity for real valued multiplicative functions), but in
\cite[Theorem~B.1]{MRT15} it was shown that for complex valued multiplicative functions  a stronger assumption is needed and  strong aperiodicity seems to be the right one.

Specializing the previous result to the case  $f_1=\cdots=f_s=f$ where $f$ is an aperiodic multiplicative function taking values plus or minus one only
(aperiodicity implies  strong aperiodicity in this case) we deduce the following:
\begin{corollary}
	Let $f\colon \mathbb{N}\to \{-1,1\}$ be an aperiodic multiplicative function which admits correlations on   $\mathbf{I}=([N_k])_{k\in\mathbb{N}}$, $N_k\to\infty$, for logarithmic averages. Then the Furstenberg system  induced by $f$
	and   $\mathbf{I}$ for logarithmic averages  is ergodic if and only if it is Bernoulli.
\end{corollary}

\subsubsection{Ergodicity and local  uniformity}
The key step taken in this article  in order to prove   Theorem~\ref{T:ErgodicChowla}, is to establish local  uniformity for the class of ergodic strongly aperiodic multiplicative functions.
The precise statement is as follows (the notions used are explained in Section~\ref{S:notation}):
\begin{theorem}\label{T:uniformity}
	Let  $f\in \mathcal{M}$ be a strongly aperiodic  multiplicative function which is ergodic
	for Ces\`aro (or logarithmic)  averages on $\mathbf{I}=([N_k])_{k\in\mathbb{N}}$, $N_k\to\infty$.
	Then $\norm{f}_{U^s(\mathbf{I})}=\norm{f}_{U^s_*(\mathbf{I})}=0$  (corr. $\norm{f}_{U^s_{\text{log}}(\mathbf{I})}=\norm{f}_{U^s_{*,\text{log}}(\mathbf{I})}=0$) for every $s\in \mathbb{N}$.
\end{theorem}
\begin{remark}
	It is shown in \cite{FH15a} (and previously in \cite{GT10b,GT12b,GTZ12c} for the M\"obius and the Liouville function) that if $f$ is an aperiodic multiplicative function, then
	for every $s\in \mathbb{N}$ we have $\lim_{N\to\infty}\norm{f_N}_{U^s(\mathbb{Z}_N)}=0$ where $\norm{\cdot}_{U^s(\mathbb{Z}_N)}$ are the Gowers uniformity norms and with $f_N$ we denote the periodic extension of $f\cdot \one_{[N]}$ to $\mathbb{Z}_N$. It should be stressed though, that when $\mathbf{I}=([N])_{N\in\mathbb{N}}$,  the  local uniformity condition $\norm{f}_{U^s(\mathbf{I})}=0$ is strictly stronger  and cannot be inferred from Gowers uniformity for any $s\geq 2$. For example, the (non-ergodic) sequence  $a(n)=\sum_{k=1}^\infty (-1)^k\,  \one_{[k^2,(k+1)^2)}(n)$, $n\in\mathbb{N}$, satisfies    $\lim_{N\to\infty} \norm{a_N}_{U^s(\mathbb{Z}_N)}=0$ for every $s\in \mathbb{N}$ where $a_N$ is defined as above, but $\norm{a}_{U^1(\mathbf{I})}=1$.  Moreover, the sequence
	$b(n)=\sum_{k=1}^\infty (-1)^{n+k}\,  \one_{[k^2,(k+1)^2)}(n)$, $n\in\mathbb{N}$,
	is  ergodic for Ces\`aro averages on ${\bf I}$  and satisfies    $\lim_{N\to\infty} \norm{b_N}_{U^s(\mathbb{Z}_N)}=0$ for every $s\in \mathbb{N}$,
	but  $\norm{b}_{U^2(\mathbf{I})}=1$.
\end{remark}
For a sketch of the proof of Theorem~\ref{T:uniformity}  see Section~\ref{SS:sketch}.

The link between   Theorem~\ref{T:ErgodicChowla} and Theorem~\ref{T:uniformity}
is given by the following result of Tao (it follows
from \cite[Theorem~1.8 and Remarks~1.9, 3.4]{T16}):
\begin{theorem}[Tao~\cite{T16}]\label{T:Tao}
	Let $s\in \mathbb{N}$,   $f$  be the Liouville or the  M\"obius function, and suppose that $f$ admits correlations for logarithmic averages  on $\mathbf{I}=([N_k])_{k\in\mathbb{N}}$, $N_k\to \infty$. If $\norm{f}_{U^{s}_{*,\text{log}}(\mathbf{I})}=0$,
	then $f$  satisfies the logarithmic Chowla  conjecture on $\mathbf{I}$ for correlations involving $s+1$ terms.
\end{theorem}
\begin{remarks}
	$\bullet$ The equivalence is proved in \cite{T16} only when $N_k=k$, $k\in\mathbb{N}$, but
	the argument in \cite{T16} also gives the  stated result.
	
	$\bullet$ An extension of this result that covers more general multiplicative functions is
	suggested in \cite[Remarks~1.10 and 3.5]{T16}. We give a related result in Theorem~\ref{T:Tao'} below.

	$\bullet$ The two main ingredients used in the  proof of Theorem~\ref{T:Tao}  were a newly devised ``entropy decrement''
	argument from \cite{Tao15} and  the Gowers uniformity of the $W$-tricked  von Mangoldt function established in   \cite{GT10b, GT12, GTZ12c}.
\end{remarks}

In order to obtain Theorem~\ref{T:ErgodicElliott} we use the following variant of the previous result which is established in Section~\ref{SS:Tao'}. The starting point of the proof is an identity for general sequences (see Proposition~\ref{P:Tao}) which is implicit
in \cite{Tao15}.
\begin{theorem}\label{T:Tao'}
	Let $f_1\in \mathcal{M}$ be a multiplicative function which admits correlations  for logarithmic averages on  $\mathbf{I}=([N_k])_{k\in\mathbb{N}}$, $N_k\to \infty$, and
	satisfies  $\norm{f_1}_{U^{s}_{\text{log}}(\mathbf{I})}=0$ for some $s\geq 2$. Then
	$$
	\mathbb{E}^{\log}_{m\in \mathbf{I}}\,  f_1(m+n_1)\cdots f_s(m+n_s)=0
	$$
	holds for all $f_2,\ldots, f_s\in \mathcal{M}$ and  distinct $n_1,\ldots, n_s\in \mathbb{N}$.
\end{theorem}

\subsection{A problem} The previous results motivate the following problem:
\begin{problem*}
	Let   $f\in \mathcal{M}$ be a strongly aperiodic  multiplicative function which admits correlations
	for Ces\`aro (or logarithmic) averages on $\mathbf{I}=([N_k])_{k\in\mathbb{N}}$, $N_k\to\infty$. Show that  the sequence
	$(f(n))_{n\in\mathbb{N}}$ is ergodic on $\mathbf{I}$ for Ces\`aro (corr. logarithmic) averages.
\end{problem*}
\begin{remark}
	In fact, it seems likely that every real valued bounded multiplicative function is ergodic for Ces\`aro averages on ${\bf I}=([N])_{N\in\mathbb{N}}$.
\end{remark}
A solution to this problem  for logarithmic averages for the Liouville (or the M\"obius)  function,   coupled with Corollary~\ref{C:ErgodicLogChowla},
would imply that the Liouville (or the M\"obius)  function satisfies the Chowla  conjecture, and hence the Sarnak conjecture, for logarithmic averages. It would also imply that all possible sign patterns are taken by consecutive values of $\lambda$, and that   each
size $\ell$ pattern  with logarithmic density $2^{-\ell}$, and as a consequence, with upper natural density greater than
$2^{-\ell}$.

Currently, we cannot even  exclude the (unlikely) possibility that $\lambda$  is generic for a   measure on $\{-1,1\}^\mathbb{N}$  which induces   a measure preserving
system with
ergodic components circle rotations.

\subsection*{Acknowledgement} I would like to thank B.~Host for various useful discussions
during the preparation of this article and B.~Kra and P.~Sarnak for useful remarks.

\section{Background,  notation, and tools}\label{S:notation}
In this section we define some concepts used throughout the article.

\subsection{Ces\`aro and logarithmic averages}\label{SS:averages}
Recall that for  $N\in\mathbb{N}$  we let  $[N]=\{1,\dots,N\}$.
If $A$ is a finite non-empty subset of $\mathbb{N}$ and $a\colon A\to \mathbb{C}$, then
we define the
\begin{itemize}
	\item
	{\em  Ces\`aro average of $(a(n))_{n\in A}$ on $A$} to be
	$$
	\mathbb{E}_{n\in A}a(n):=\frac{1}{|A|}\sum_{n\in A}a(n);
	$$
	
	\item	
	{\em logarithmic average   of $(a(n))_{n\in A}$ on $A$}  to be
	$$
	\mathbb{E}^{\log}_{n\in A}\, a(n):= \frac{1}{\sum_{n\in A} \frac{1}{n}} \sum_{n\in A} \frac{a(n)}{n}.
	$$
\end{itemize}

We say that the sequence of intervals
$\mathbf{I}=(I_N)_{N\in\mathbb{N}}$ is a {\em   F{\o}lner sequence for}
\begin{itemize}
	\item  {\em    Ces\`aro   averages} if $\lim_{N\to\infty} |I_N|=\infty$;
	
	\item  {\em     logarithmic averages} if
	$
	\lim_{N\to\infty} \sum_{n\in I_N} \frac{1}{n}=\infty.
	$
\end{itemize}
If $I_N=(a_N,b_N)$, $N\in\mathbb{N}$, the second condition is equivalent to  $b_N/a_N\to \infty$.

If $a\colon \mathbb{N}\to \mathbb{C}$ is a bounded sequence, and $\mathbf{I}=(I_N)_{N\in\mathbb{N}}$ is a F{\o}lner sequence of intervals for Ces\`aro or logarithmic averages,  we define the
\begin{itemize}
	\item {\em Ces\`aro mean  of $(a(n))_{n\in\mathbb{N}}$ on $\mathbf{I}$} to be
	$$
	\mathbb{E}_{n\in \mathbf{I}}\, a(n):=\lim_{N\to\infty} \mathbb{E}_{n\in I_N}\, a(n)
	$$
	if the limit exists;
	
	\item {\em logarithmic mean  of $(a(n))_{n\in\mathbb{N}}$ on $\mathbf{I}$} to be
	$$
	\mathbb{E}^{\log}_{n\in \mathbf{I}}\, a(n):=\lim_{N\to\infty} \mathbb{E}^{\log}_{n\in I_N}\, a(n)
	$$
	if the limit exists.
	
	\item If the previous mean values exist
	for every F{\o}lner sequence of intervals $\mathbf{I}$,
	then we denote the common mean value by $\mathbb{E}_{n\in\mathbb{N}}\, a(n)$ and
	$\mathbb{E}^{\log}_{n\in\mathbb{N}}\, a(n)$ respectively.
\end{itemize}
Note that  all  these mean values are  shift invariant, meaning, for every $a\in \ell^\infty(\mathbb{N})$ and $h\in\mathbb{N}$ the sequences $(a(n))_{n\in\mathbb{N}}$ and $(a(n+h))_{n\in\mathbb{N}}$
have the same Ces\`aro/logarithmic mean on $\mathbf{I}$.

It is easy to see  using partial summation, that if $(a(n))_{n\in\mathbb{N}}$ has a mean value on $\mathbf{I}=([N])_{N\in\mathbb{N}}$, then its logarithmic mean on $\mathbf{I}=([N])_{N\in\mathbb{N}} $ exists
and the two means coincide. Note though that a similar statement fails for some subsequences $\mathbf{I}=([N_k])_{k\in\mathbb{N}}$ with $N_k\to \infty$, and also the converse implication fails for $\mathbf{I}=([N])_{N\in\mathbb{N}} $.

\begin{definition}
	Let $\mathbf{I}=(I_N)_{N\in\mathbb{N}}$ be a   sequence of intervals with $|I_N|\to \infty$.
	We say that the sequence $a\in\ell^\infty(\mathbb{N})$ satisfies the {\em Chowla conjecture for Ces\`aro  averages on $\mathbf{I}$} if
	$$
	\mathbb{E}_{m\in \mathbf{I}}\,  a(c_1m+n_1)\cdots a(c_sm+n_s)=0,
	$$
	for all $s\in \mathbb{N}$, $c_1,\ldots, c_s\in \mathbb{N}$, and non-negative integers $n_1,\ldots, n_s$
	such that $c_i n_j\neq c_jn_i$ for all $i\neq j$.

	Similar definitions apply for logarithmic averages and when we restrict the number of terms in the product.
\end{definition}

\subsection{Measure preserving systems}
A {\em measure preserving system}, or simply {\em a system}, is a quadruple $(X,\mathcal{X},\mu,T)$
where $(X,\mathcal{X},\mu)$ is a probability space and $T\colon X\to X$ is an invertible  measure preserving transformation. The system is ergodic if the only sets that
are  invariant by $T$ have measure $0$ or $1$. The von Neumann ergodic theorem implies  that for ergodic systems we have
$$
\lim_{N\to \infty}\mathbb{E}_{n\in I_N} \int T^nF\cdot G\, d\mu= \int F\, d\mu\, \int G\, d\mu,
$$
for every  sequence of intervals $(I_N)_{N\in \mathbb{N}}$
with $|I_N|\to \infty$
and functions $F,G\in L^2(\mu)$. In the previous statement and throughout,  with $TF$ we denote the composition $F\circ T$
and with $T^n$ we denote the composition $T\circ\cdots\circ T$.

\subsection{Ergodicity of sequences}\label{SS:ergodic}
To each bounded sequence  that is distributed ``regularly'' along a  sequence of intervals with lengths increasing to infinity,
we  associate a measure preserving system; the notion of ergodicity of this sequence is then naturally inherited from the corresponding property of the system.

\begin{definition}\label{D:correlations}
	Let $ {\bf I}:=(I_N)_{N\in\mathbb{N}}$ be a sequence of intervals with $|I_N|\to \infty$.
	We say that the   sequence $a\in \ell^\infty(\mathbb{N})$  {\em admits correlations for Ces\`aro averages   on ${\bf I}$}, if the limit
	$$
	\lim_{N\to\infty} \mathbb{E}_{m\in I_N}\,  b_1(m+n_1)\cdots b_s(m+n_s)
	$$
	exists, for every $s\in \mathbb{N}$,   $n_1,\ldots, n_s\in \mathbb{N}$ (not necessarily distinct), and all sequences
	$b_1,\ldots, b_s$  that
	belong to the set $\{a, \bar{a}\}$.
	
	A similar definition applies for logarithmic averages;  in place of $\mathbb{E}_{m\in I_N}$  use $\mathbb{E}^{\log}_{m\in I_N}$.
\end{definition}
\begin{remark}
	If $a\in \ell^\infty(\mathbb{Z})$, then using a diagonal argument we get that any sequence of intervals $\mathbf{I}=(I_N)_{N\in \mathbb{N}}$
	has a subsequence $\mathbf{I}'=(I_{N_k})_{k\in\mathbb{N}}$, such that the sequence  $(a(n))_{n\in\mathbb{N}}$ admits correlations on $\mathbf{I}'$.
\end{remark}

The correspondence principle of Furstenberg was originally used in \cite{Fu77} in order to translate  Szemer\'edi's theorem on arithmetic progressions to an ergodic statement.
We will use the following  variant
which applies to general  bounded sequences:
\begin{proposition}\label{P:correspondence}
	Let $a\in\ell^\infty(\mathbb{N})$ be a sequence that  admits
	correlations for Ces\`aro averages  on the sequence of intervals ${\bf I}:=(I_N)_{N\in\mathbb{N}}$ with $|I_N|\to \infty$.
	Then there exist a  system $(X,\mathcal{X},\mu, T)$ and a
	function $F\in L^\infty(\mu)$, such that
	$$
	\mathbb{E}_{m\in {\mathbf{I}} }\,  a_1(m+n_1)\cdots a_s(m+n_s)= \int   T^{n_1}
	F_1\cdots T^{n_s}
	F_s\, d\mu,
	$$
	for every $s\in \mathbb{N}$,  $n_1, \ldots, n_s\in \mathbb{N}$,
	where for $j=1,\ldots, s$ the sequence
	$a_j$ is either  $a$ or $\overline{a}$
	and $F_j$ is  $F$ or $\overline{F}$ respectively.
	A similar statement holds for logarithmic averages.
\end{proposition}
\begin{remark}
	For sequences  bounded by $1$, in the previous correspondence,  $X$, $\mathcal{X}$, $T$, and $F$ can be taken to be independent of the sequence $a$ and ${\bf I}$, and it is only the measure $\mu$ that varies. Furthermore,  the system constructed is uniquely determined up to isomorphism by the pair $(a,\mathbf{I})$.
\end{remark}
\begin{proof}
	Let $X:=D^{\mathbb{Z}}$, where
	$D$ is the closed disk in $\mathbb{C}$ of radius $\norm{a}_\infty$,
	be endowed with the product topology and with the invertible and continuous shift $T$ given by
	$
	(Tx)(k)=x(k+1), k\in \mathbb{Z}.$
	We define $F\in C(X) $ by $F(x):=x(0)$, $x\in X$,
	and  	$\omega \in D^{\mathbb{Z}}$
	by $\omega(k):=a(k)$ for  $k\in\mathbb{N}$
	and $\omega(k)=0$ for $k\leq 0$.
	Lastly, we let $\mu$ be a $w^*$-limit point for   the sequence of measures
	$
	\mu_N:=\frac 1{|I_N|}\sum_{n\in I_N}\delta_{T^n \omega},  N\in \mathbb{N}.
	$
	Then $\mu$ is a $T$-invariant probability measure on $X$, and since	$F(T^n \omega)=a(n)$ for $n\in\mathbb{N}$ and $(a(n))_{n\in\mathbb{N}}$ admits correlations for Ces\`aro averages on $\mathbf{I}$, the asserted
	identity follows immediately.
\end{proof}	

\begin{definition}
	Let $a\in\ell^\infty(\mathbb{N})$ be a sequence that  admits
	correlations  for Ces\`aro averages  on the sequence  of intervals ${\bf I}:=(I_N)_{N\in\mathbb{N}}$
	with $|I_N|\to \infty$.
	We call the system defined in Proposition~\ref{P:correspondence}
	{\em the Furstenberg  system induced by $a$ and $\mathbf{I}$ for Ces\`aro averages}.
	
	A similar definition applies for logarithmic averages.
\end{definition}
\begin{remarks}
	$\bullet$		A priori a sequence $a\in \ell^\infty(\mathbb{Z})$ may have uncountably many non-isomorphic Furstenberg systems
	depending on which sequence of intervals ${\bf I}$ we choose to work with. Furthermore, for fixed $(a,{\bf I})$ the  Furstenberg systems associated with Ces\'aro and  logarithmic averages could be very different.

	$\bullet$	If we assume that the Liouville function admits correlations  on $([N])_{N\in\mathbb{N}}$,
	then  the corresponding Furstenberg system is {\em the Liouville system} alluded to in the introduction.
\end{remarks}

\begin{definition} Let $\mathbf{I}=(I_N)_{N\in\mathbb{N}}$ be a sequence of intervals with $|I_N|\to \infty$.
	We say that a sequence  $a\in \ell^\infty(\mathbb{N})$ is {\em  ergodic  for Ces\`aro averages   on   $\mathbf{I}$} if
	\begin{enumerate}
		\item it admits  correlations    for Ces\`aro averages  on $\mathbf{I}$; and
		
		\item the induced measure preserving system  for Ces\`aro averages  is ergodic.
	\end{enumerate}
	A similar definition applies for logarithmic averages in which case  we say that  $a\in \ell^\infty(\mathbb{N})$ is {\em  ergodic  for logarithmic averages  on   $\mathbf{I}$}.
\end{definition}
Note that the second condition  for Ces\`aro averages is equivalent
to having the identities
\begin{equation}\label{E:ET}
	\lim_{N\to\infty} \mathbb{E}_{n\in [N]}\big( \mathbb{E}_{m\in \mathbf{I}}\,  b(m+n)\cdot c(m)\big)=
	\mathbb{E}_{m\in\mathbf{I}}\, b(m) \cdot \mathbb{E}_{m\in\mathbf{I}}\, c(m),
\end{equation}
for all $b,c\in \ell^\infty(\mathbb{N})$ of the form
$b(m)=a_1(m+h_1)\cdots a_s(m+h_s)$,  $m\in\mathbb{N}$,  for some $s\in \mathbb{N}$,  non-negative integers $h_1,\ldots, h_s$, and $a_i\in \{a,\overline{a}\}$, and  similarly for $(c(m))_{m\in\mathbb{N}}$.
For logarithmic averages a similar condition holds with $\mathbb{E}_{m\in \mathbf{I}}$ replaced by
$\mathbb{E}^{\log}_{m\in \mathbf{I}}$.

\subsection{Ergodic seminorms and the factors $\mathcal{Z}_s$.}\label{SS:GHKseminorms}
Following \cite{HK05a}, if $(X,\mathcal{X}, \mu, T)$ is a system we define  the Host-Kra seminorms of $F\in L^\infty(\mu)$   inductively by
$$
\nnorm{F}^2_{1}:=\mathbb{E}_{h\in\mathbb{N}} \int T^hF\cdot \overline{F} \, d\mu, \qquad
\nnorm F_{s+1}^{2^{s+1}} :=\mathbb{E}_{h\in\mathbb{N}}
\nnorm{T^hF\cdot \overline{F}}_{s}^{2^{s}},
$$
for $s\in \mathbb{N}$, where the implicit  limits defining the mean values $\mathbb{E}_{h\in\mathbb{N}}$
are known to exist by \cite{HK05a}. It is also shown in the same article that for every $s\in \mathbb{N}$ there exist  $T$-invariant sub-$\sigma$-algebras $\mathcal{Z}_s$ of $L^\infty(\mu)$
such that $\nnorm{F}_{s+1}=0$ if and only if $F\, \bot\,  L^2(\mathcal{Z}_s)$, meaning $\int F\cdot G\, d\mu=0$ for every $G\in L^2(\mu)$ that is $\mathcal{Z}_s$-measurable. We remark that the previous properties are proved in \cite{HK05a} only for ergodic systems but similar
arguments apply for general systems.
We are going to use the following important structure theorem (nilsystems are defined in Section~\ref{SS:nilmanifolds}):

\begin{theorem}[Host, Kra \cite{HK05a}]\label{T:HK}
	Let $(X,\mathcal{X}, \mu,T)$ be an ergodic system and $s\in \mathbb{N}$. Then the system $(X,\mathcal{Z}_s,\mu,T)$ is an inverse limit of $s$-step nilsystems.
\end{theorem}
The last property means that there exist $T$-invariant sub-$\sigma$-algebras $\mathcal{Z}_{s,n}$, $n\in\mathbb{N}$,   that span $\mathcal{Z}_s$, such that for every $n\in\mathbb{N}$ the factor system associated with $\mathcal{Z}_{s,n}$ is isomorphic to an $s$-step nilsystem.

\subsection{Local  uniformity seminorms}\label{SS:uniformity}
Let $\mathbf{I}=(I_N)_{N\in\mathbb{N}}$ be a  sequence of intervals with $|I_N|\to \infty$ and  $a\in\ell^\infty(\mathbb{N})$ be
a sequence that admits correlations for Ces\`aro averages on $\mathbf{I}$.
Following \cite{HK09},  we define   the uniformity seminorms  $\norm{a}_{U^s(\mathbf{I})}$ inductively
as follows:
$$
\norm{a}_{U^1(\mathbf{I})}^2:=\mathbb{E}_{h\in\mathbb{N}}
\big(\mathbb{E}_{n\in \mathbf{I}}\, a(n+h)\, \overline{a(n)})
$$
(by van der Corput's lemma $
|\mathbb{E}_{n\in \mathbf{I}}a(n)|\leq 4\,  \norm{a}_{U^1(\mathbf{I})}$)
and for $s\in \mathbb{N}$
$$
\norm{a}_{U^{s+1}(\mathbf{I})}^{2^{s+1}}:=\mathbb{E}_{h\in \mathbb{N}} \norm{S_ha\cdot \overline{ a}}_{U^s(\mathbf{I})}^{2^s},
$$
where $S_h$ is the shift operator on $ \ell^\infty(\mathbb{N})$ defined for $h\in\mathbb{N}$  by
$$
(S_ha)(n):=a(n+h), \quad n\in\mathbb{N}.
$$
It is not immediately clear that all the iterative limits defining the above averages exist. This can be proved by reinterpreting these seminorms in ergodic terms  using
the measure preserving  system $(X,\mathcal{X}, \mu, T)$
and the function  $F\in L^\infty(\mu)$ induced by $a$ and $\mathbf{I}$;  we then have
$\norm{a}_{U^s(\mathbf{I})}=\nnorm{F}_{s} $
where $\nnorm{F}_{s}$ is defined as in Section~\ref{SS:GHKseminorms}.
Using the ergodic reinterpretation and \cite[Theorem~1.2]{HK05a} we deduce the identity
\begin{equation}\label{E:Us}
	\norm{a}_{U^s(\mathbf{I})}^{2^s}=\mathbb{E}_{\underline{h}\in \mathbb{N}^s}\Big(\mathbb{E}_{n\in\mathbf{I}}\prod_{\epsilon\in [\![s]\!]}\mathcal{C}^{|\epsilon|} a(n+\epsilon\cdot \underline{h})\Big),
\end{equation}
where
$$
\mathcal{C}^k(z):=\begin{cases} z\quad \text{ if } k \text{ is even};  \\
\overline{z} \quad \text{ if } k \text{ is odd}, \end{cases}
$$
and for $s\in \mathbb{N}$ we let $\underline{0}:=(0,\ldots, 0)$,
$$
[\![s]\!]:=\{0,1\}^{s}, \quad [\![s]\!]^*:=[\![s]\!]\setminus \{\underline{0}\},
$$
and
for $\epsilon=(\epsilon_1,\ldots, \epsilon_s)$  we let
$$
|\epsilon|:=\epsilon_1+\cdots+\epsilon_s.
$$
Furthermore, the limit $\mathbb{E}_{\underline{h}\in \mathbb{N}^s}$ can be defined using averages taken  over arbitrary  F{\o}lner sequences of subsets of $\mathbb{N}^s$, or can be defined using the iterative limit $\mathbb{E}_{h_s\in \mathbb{N}}\cdots \mathbb{E}_{h_1\in \mathbb{N}}$.
All these limits exist and are equal; this follows from \cite[Theorem~1.2]{HK05a}.
It is shown in \cite{HK05a} that    $\nnorm{F}_{s} \leq \nnorm{F}_{s+1}$ for every $F\in L^\infty(\mu)$ and $s\in \mathbb{N}$;  we deduce that
$$
\norm{a}_{U^s(\mathbf{I})}\leq  \norm{a}_{U^{s+1}(\mathbf{I})}, \quad \text{for every } s\in \mathbb{N}.
$$

In a similar fashion, if   $a\in\ell^\infty(\mathbb{N})$
admits correlations for logarithmic averages on $\mathbf{I}$,
we define the uniformity seminorms for logarithmic averages $\norm{a}_{U^s_\text{log}(\mathbf{I})}$
as follows:
$$
\norm{a}^2_{U^1_\text{log}(\mathbf{I})}:=\mathbb{E}_{h\in\mathbb{N}}\big( \mathbb{E}^{\log}_{n\in \mathbf{I}}\, a(n+h)\, \overline{a(n)}\big)
$$
and for $s\in \mathbb{N}$
$$
\norm{a}_{U^{s+1}_\text{log}(\mathbf{I})}^{2^{s+1}}:=\mathbb{E}_{h\in \mathbb{N}} \norm{S_ha\cdot \overline{ a}}_{U^s_\text{log}(\mathbf{I})}^{2^s}.
$$
All  implicit limits   defining the  mean values $\mathbb{E}_{h\in\mathbb{N}}$  can be shown to exist. Note  that in the definition of the uniformity seminorms for logarithmic averages only the inner-most average is logarithmic, the others can be given by any shift invariant averaging scheme we like.
For example, we have
$$
\norm{a}_{U^2_\text{log}(\mathbf{I})}^4=\mathbb{E}_{(h_1,h_2)\in \mathbb{N}^2}
\big(\mathbb{E}^{\log}_{n\in \mathbf{I}}\,  a(n+h_1+h_2)\cdot \overline{a(n+h_1)}\cdot \overline{a(n+h_2)}\cdot a(n)\big).
$$

We also use variants of some local uniformity seminorms  introduced by Tao in \cite{T16} when $\mathbf{I}=([N])_{N\in\mathbb{N}}$. For $s\in \mathbb{N}$, $a\in \ell^\infty(\mathbb{N})$, and  F{\o}lner sequence of intervals $\mathbf{I}$,  we let
$$
\norm{a}_{U^{s}_*(\mathbf{I})}:=\limsup_{H\to\infty}\limsup_{N\to\infty}\mathbb{E}_{n\in I_N} \norm{S_na}_{U^s[H]}
$$
(where $(S_na)(m):=a(n+m)$) and
$$
\norm{a}_{U^{s}_{*,\text{log}}(\mathbf{I})}:=\limsup_{H\to\infty}\limsup_{N\to\infty}\mathbb{E}^{\log}_{n\in I_N} \norm{S_na}_{U^s[H]}.
$$
We used that
$$
\norm{a}_{U^s[N]}:=
\frac{\norm{a\cdot {\bf 1}_{[N]}}_{U^s(\mathbb{Z}_{2N})}}{\norm{\one_{[N]} }_{U^s(\mathbb{Z}_{2N})}} ,
$$
(\cite[Lemma~A.2]{FH15a} explains why $2N$ is an appropriate choice) where
$\mathbb{Z}_N:=\mathbb{Z}/(N\mathbb{Z})$ and $\norm{a}_{U^s(\mathbb{Z}_N)}$ are the Gowers uniformity norms. These were defined in \cite{Gow01} as follows:
$$
\norm{a}_{U^{1}(\mathbb{Z}_N)}:=|\mathbb{E}_{m\in \mathbb{Z}_N}a(m)|
$$
and for $s\in \mathbb{N}$
$$
\norm{a}_{U^{s+1}(\mathbb{Z}_N)}^{2^{s+1}}:=\mathbb{E}_{h\in \mathbb{Z}_N} \norm{S_ha\cdot \overline{a}}_{U^s(\mathbb{Z}_N)}^{2^s},
$$
where for  $N\in \mathbb{N}$ we  use the periodic extension of $a\cdot {\bf 1}_{[N]}$ to $\mathbb{Z}_N$ in the previous computations, or equivalently, we
define $(S_ha)(n):=a(n+h\! \! \mod{N})$ for $n\in \mathbb{Z}_N$.

\begin{proposition}\label{E:normequiv}
	Let  $s\in \mathbb{N}$.
	If  $a\in \ell^\infty(\mathbb{N})$ is a sequence that admits correlations for Ces\`aro averages on the  sequence
	of intervals $\mathbf{I}$,  then
	$$
	\norm{a}_{U^{s}_*(\mathbf{I})}\leq 4\, \norm{a}_{U^{s}(\mathbf{I})}.
	$$
	A similar statement holds for logarithmic averages on $\mathbf{I}$ and the corresponding estimate is
	$
	\norm{a}_{U^{s}_{*,\text{log}}(\mathbf{I})}\leq 4\, \norm{a}_{U^{s}_\text{log}(\mathbf{I})}$.
\end{proposition}
\begin{proof}
	Let $H\in \mathbb{N}$ and $a\colon \mathbb{Z}_H\to \mathbb{C}$ be  bounded by $1$. Then  arguing as in the proof of Proposition~3.2 in \cite{CF11}, we get that for  every $H_1,\ldots, H_s\in \mathbb{N}$ the following estimate holds
	\begin{equation}\label{E:est}
		\norm{a}_{U^s(\mathbb{Z}_H)}^{2^s}\leq 2^s\, \mathbb{E}_{h_1\in [H_1]}\cdots \mathbb{E}_{h_s\in [H_s]} \Big(\mathbb{E}_{h\in [H]}\prod_{\epsilon\in [\![s]\!]}\mathcal{C}^{|\epsilon|} a(h+\epsilon\cdot \underline{h})\Big)+\sum_{j=1}^sH_i^{-1}.
	\end{equation}
	
	Now let $a\colon \mathbb{N}\to \mathbb{C}$ be bounded by $1$. For $H\in \mathbb{N}$ we define
	$a_H\colon  \mathbb{Z}_{2H}\to \mathbb{C}$ by $a_H(n):=a(n)\cdot \one_{[H]}(n)$ for $n\in [2H]$ and we extend it periodically to $\mathbb{Z}_{2H}$.
	Using the definition of $\norm{a}_{U^s[H]}$ in conjunction with the estimate  \eqref{E:est}, applied to  $a_H$, and using that 	
	$\norm{{\bf 1}_{[H]}}_{U^s(\mathbb{Z}_{2H})}\geq \norm{{\bf 1}_{[H]}}_{U^1(\mathbb{Z}_{2H})}
	=\frac{1}{2}$, we deduce  that for  every $H, H_1,\ldots, H_s\in \mathbb{N}$ we have
	$$
	2^{-2^s}\norm{a}_{U^s[H]}^{2^s}\leq 2^{s}\, \mathbb{E}_{h_1\in [H_1]}\cdots \mathbb{E}_{h_s\in [H_s]} \Big(\mathbb{E}_{h\in [2H]}\prod_{\epsilon\in [\![s]\!]}\mathcal{C}^{|\epsilon|} a_H(h+\epsilon\cdot \underline{h})\Big)+\sum_{j=1}^sH_i^{-1},
	$$
	where $\underline{h}=(h_1,\ldots, h_s)$.
	Since  for $h\in [2H]$ and $h_1\in [H_1],\ldots, h_s\in [H_s]$ we have
	$a_H(h+\epsilon\cdot \underline{h})=
	(a\cdot{\bf 1}_{[H]})(h+\epsilon\cdot \underline{h})$ except possibly when
	$h\in[H-(H_1+\cdots+H_s), H]\cup [2H-(H_1+\cdots+H_s), 2H]$, we get
	$$
		2^{-2^s}\norm{a}_{U^s[H]}^{2^s}\leq 2^{s}\, \mathbb{E}_{h_1\in [H_1]}\cdots \mathbb{E}_{h_s\in [H_s]} \Big(\mathbb{E}_{h\in [H]}\prod_{\epsilon\in [\![s]\!]}\mathcal{C}^{|\epsilon|} a(h+\epsilon\cdot \underline{h})\Big)+\sum_{j=1}^sH_i^{-1} +2^sH^{-1}\, \sum_{j=1}^s H_j,
	$$
	where the sums $h+\epsilon\cdot \underline{h}$
	are taken in $\mathbb{N}$.
	Using this estimate for the sequence $S_na$,
	averaging over $n\in I_N$, taking $N\to \infty$, and then making the change of   variables $n\mapsto n-h$,  we get that
	for every $H, H_1,\ldots, H_s\in \mathbb{N}$ we have
	\begin{multline*}
		2^{-2^s}\limsup_{N\to\infty}\mathbb{E}_{n\in I_N} \norm{S_na}_{U^s[H]}^{2^s}\leq\\
		2^{s}\, \mathbb{E}_{h_1\in [H_1]}\cdots \mathbb{E}_{h_s\in [H_s]} \Big(\mathbb{E}_{n\in\mathbf{I}}\prod_{\epsilon\in [\![s]\!]}\mathcal{C}^{|\epsilon|} a(n+\epsilon\cdot \underline{h})\Big)+ \sum_{j=1}^sH_i^{-1} +2^{s}H^{-1}\, \sum_{j=1}^s H_j.
	\end{multline*}
	Finally, recall that  $
	\norm{a}_{U^{s}_*(\mathbf{I})}=\limsup_{H\to\infty}
	\limsup_{N\to\infty}\mathbb{E}_{n\in I_N} \norm{S_na}_{U^s(\mathbb{Z}_H)}
	$. Thus,  if on the last estimate  we take $H\to\infty$ and then let $H_s\to \infty,\ldots,H_1\to\infty,$ we get that
	$$
	2^{-2^s}\norm{a}_{U^{s}_*(\mathbf{I})}^{2^s}\leq 2^{s}\, \mathbb{E}_{h_1\in \mathbb{N}}\cdots \mathbb{E}_{h_s\in \mathbb{N}}
	\Big(\mathbb{E}_{n\in \mathbf{I}}  \prod_{\epsilon\in [\![s]\!]}\mathcal{C}^{|\epsilon|} a(n+\epsilon\cdot \underline{h})\Big)=2^{s}\,\norm{a}_{U^{s}(\mathbf{I})}^{2^s}.
	$$
	This proves the first   estimate. The proof of the second estimate
	is similar.
\end{proof}

\subsection{Multiplicative functions and strong aperiodicity} \label{SS:aperiodic}
A function $f\colon \mathbb{N}\to \mathbb{C}$ is called \emph{multiplicative} if
$$
f(mn)=f(m)f(n) \ \text{ whenever } \  (m,n)=1.
$$
It is called \emph{completely multiplicative} if  the previous identity holds for every $m,n\in \mathbb{N}$.
We   let
$$
\mathcal{M}:=\{f\colon \mathbb{N}\to \mathbb{C} \text{ is multiplicative, bounded, and  } |f(p)|=1 \text{ for every } p\in \P\}.
$$
A \emph{Dirichlet character} is a periodic
completely multiplicative function $\chi$  with $\chi(1)=1$.	
We say that   $f\in \mathcal{M}$ is \emph{aperiodic} (or \emph{non-pretentious} following \cite{GS16})
if
$$\lim_{N\to\infty}\mathbb{E}_{n\in [N]}f(an+b)=0, \quad \text{  for all } a,b\in \mathbb{N},
$$
or equivalently, if
$\lim_{N\to\infty}\mathbb{E}_{n\in [N]}f(n)\, \chi(n)=0$ for every Dirichlet character $\chi$.

The uniformity result stated in  Theorem~\ref{T:uniformity}
holds for a class of multiplicative functions that satisfy a condition introduced in \cite{MRT15}
which is somewhat stronger than aperiodicity. In order to state it  we need the notion of the  distance between two multiplicative functions defined as in  \cite{GS16}:
\begin{definition}
	Let $\P$ be the set of primes.
	We let $\D\colon \mathcal{M}\times \mathcal{M}\to [0,\infty]$ be given by
	$$
	\D(f,g)^2=\sum_{p\in \P} \frac{1}{p}\,\big(1-\Re\big(f(p) \overline{g(p)}\big)\big)
	$$
	where $\Re(z)$ denotes the real part of a compelx number $z$. 
	We also let $\D\colon \mathcal{M}\times \mathcal{M}\times \mathbb{N} \to [0,\infty]$ be given by
	$$
	\D(f,g;N)^2=\sum_{p\in \P\cap [N]} \frac{1}{p}\,\bigl(1-\Re\bigl(f(p) \overline{g(p)}\bigr)\bigr)
	$$
	and $M\colon \mathcal{M}\times \mathbb{N} \to [0,\infty)$ be given by
	$$
	M(f;N):=\min_{|t|\leq N} \D(f, n^{it};N)^2.
	$$
\end{definition}

A celebrated  theorem of Hal\'asz~\cite{Hal68}
states that a multiplicative function  $f\in \mathcal{M}$ has zero  mean value  if and only if
for every $t\in \R$  we either have $\D(f, n^{it})=\infty$
or
$
f(2^k)= -2^{ikt}$ for all $k\in \mathbb{N}$.
For our purposes we need information on averages of multiplicative functions  taken on typical short intervals. Such results were obtained   in \cite{MR15, MRT15}, under conditions that motivate the following definition:
\begin{definition}\label{D:uniformly}
	The multiplicative function $f\in \mathcal{M}$ is
	{\em strongly  aperiodic}
	if $M(f\cdot\chi;N)\to \infty$ as $N
	\to \infty$ for every Dirichlet character $\chi$.
\end{definition}
Note  that strong aperiodicity implies aperiodicity. The converse is
not in general true (see  \cite[Theorem~B.1]{MRT15}), but it is
true for (bounded) real valued  multiplicative functions (see  \cite[Appendix~C]{MRT15}). In particular,  the  Liouville and the M\"obius function are strongly aperiodic.
Furthermore,  if $f\in \mathcal{M}$  satisfies
\begin{enumerate}
	\item $f(p)$ is a $d$-th root of unity
	for all but finitely many primes $p$; and
	
	\item 	$\D(f, \chi)=\infty$ for every Dirichlet character $\chi$,
\end{enumerate}
then $f$ is strongly aperiodic (see \cite[Proposition~6.1]{F16}).

We will need the following result;
a quantitative variant  of which is implicit in \cite{MRT15}
(the stated version  is also deduced from
\cite[Theorem A.1]{MRT15} in \cite[Theorem~4.1]{F16}):
\begin{proposition}[Matom\"aki, Radziwi{\l}{\l}, Tao~\cite{MRT15}]
	\label{P:stronaperiodic}
	Let $f\in \mathcal{M}$ be a strongly  aperiodic multiplicative function that
	admits
	correlations  for Ces\`aro averages  on the  sequence of intervals ${\bf I}:=([N_k])_{k\in\mathbb{N}}$ with $N_k\to \infty$.
	Then
	$$
	\lim_{H\to \infty}\mathbb{E}_{h\in [H]} |\mathbb{E}_{n\in {\bf I}} \, f(n+h)\cdot \overline{f(n)} |=0.
	$$
	A similar statement also holds for logarithmic averages on ${\bf I}$.
\end{proposition}
\begin{remark}
	It follows from \cite[Theorem~B.1]{MRT15} that strong aperiodicity cannot be replaced by aperiodicity; in particular,  there exist  an aperiodic multiplicative function $f\in \mathcal{M}$, a positive constant $c$,  and a sequence of intervals ${\bf I}:=([N_k])_{k\in\mathbb{N}}$ with $N_k\to \infty$, such that
	$$
	|\mathbb{E}_{n\in {\bf I}} \, f(n+h)\cdot \overline{f(n)} |\geq c, \quad \text{ for every } h\in \mathbb{N}.
	$$	
\end{remark}

\subsection{Local uniformity implies the  Elliott conjecture} \label{SS:Tao'}
In this subsection we prove  Theorem~\ref{T:Tao'} by adapting  the argument in \cite{Tao15} which deals with the case where all the multiplicative functions  are equal to  the Liouville or the M\"obius function.
In what follows, if $(a(p))_{p\in\P}$ is a sequence indexed by the primes, we denote by
$
\mathbb{E}_{p\in\P}\, a(p)$ the limit $\lim_{N\to\infty}\frac {\log N}N\sum_{p\leq N}a(p)
$
if it exists.

Our starting point is  the following identity which is implicit in \cite{Tao15} and its proof was   sketched in \cite[Appendix~C]{FH17} (see also \cite[Theorem~3.6]{TT17} for a variant of this identity):
\begin{proposition}\label{P:Tao}
	Let ${\bf I}=([N_k])_{k\in\mathbb{N}}$ be a sequence of intervals with $N_k\to \infty$,  $(c_p)_{p\in\P}$ be a bounded sequence of complex numbers, $s\in \mathbb{N}$,  $a_1,\ldots, a_s\in \ell^\infty(\mathbb{N})$, and    $n_1,\ldots, n_s\in \mathbb{N}$.
	Then,
	assuming that on the left and  right hand side below the   limit $\mathbb{E}^{\log}_{m\in {\bf I}}$ exists for every $p\in\P$ and the limit  $\mathbb{E}_{p\in\P}$  exists, we have
	$$
	\mathbb{E}_{p\in\P}\, c_p\,
	\big(\mathbb{E}^{\log}_{m\in {\bf I}}\,\prod_{j=1}^s a_j(pm+pn_j)\big)=
	\mathbb{E}_{p\in \P}\,
	c_p\,
	\big(\mathbb{E}^{\log}_{m\in {\bf I}}\, \prod_{j=1}^s a_j(m+pn_j)\big).
	$$
\end{proposition}
We deduce from this the following identity for multiplicative functions:
\begin{corollary}\label{C:Tao}
	Let ${\bf I}=([N_k])_{k\in\mathbb{N}}$ be a sequence of intervals with $N_k\to \infty$,  $s\in \mathbb{N}$,  $f_1,\ldots, f_s \in \mathcal{M}$,  and  $n_1,\ldots, n_s\in \mathbb{N}$.
	Suppose that
	for every $p\in\P$   on the left and  right hand side below the   limit $\mathbb{E}^{\log}_{m\in {\bf I}}$ exists and 	the limit  $\mathbb{E}_{p\in \P}$ exists. Furthermore, suppose  that the limit $\mathbb{E}^{\log}_{m\in {\bf I}}\,\prod_{j=1}^s f_j(pm+pn_j)$ exists for every $p\in \P$.\footnote{This assumption can be omitted using a subsequential argument but we will not need this.} Then we have
   $$
	\mathbb{E}^{\log}_{m\in {\bf I}}\,\prod_{j=1}^s f_j(m+n_j)=
	\mathbb{E}_{p\in \P}\,
	c_p\,
	\big(\mathbb{E}^{\log}_{m\in {\bf I}}\, \prod_{j=1}^s f_j(m+pn_j)\big).
	$$
	where $c_p:=\prod_{j=1}^s \overline{f_j(p)}$ for  $p\in\P$.
\end{corollary}
\begin{proof}
	For $p\in \P$ and $j=1,\ldots s$, we have $f_j(p(m+n_j))=f_j(p) \, f_j(m+n_j)$ unless $m\equiv -n_j\pmod{p}$. Hence,
	$$
	\mathbb{E}^{\log}_{m\in {\bf I}}\,\prod_{j=1}^s f_j(m+n_j)=c_p\,  	\mathbb{E}^{\log}_{m\in {\bf I}}\,\prod_{j=1}^s f_j(pm+pn_j)+O(1/p)
	$$
	where the implicit constant depends only on $s$ and on the sup-norm of $f_1,\ldots, f_s$.
	Averaging over $p\in \P$ we get
	$$
	\mathbb{E}^{\log}_{m\in {\bf I}}\,\prod_{j=1}^s f_j(m+n_j)=\mathbb{E}_{p\in\P}\,  c_p\,  \big(\mathbb{E}^{\log}_{m\in {\bf I}}\,\prod_{j=1}^s f_j(pm+pn_j)\big).
	$$
	Applying Proposition~\ref{P:Tao}, we get the asserted identity.
\end{proof}
We will also need the following multiple ergodic theorem:
\begin{proposition}\label{P:ergodic}
	Let $(X,\mathcal{X}, \mu,T)$ be a system, $s\geq 2$,  $F_1,\ldots, F_s\in L^\infty(\mu)$, and $n_1,\ldots, n_s\in \mathbb{N}$ be distinct integers.
	Suppose that $\nnorm{F_1}_s=0$. Then
	$$
	\mathbb{E}_{p\in \P}\Big|\int \prod_{j=1}^s T^{pn_j}F_j\, d\mu\Big|=0.
	$$
\end{proposition}
\begin{proof}
	It suffices to show that
	$$
	\mathbb{E}_{p\in \P}\int \prod_{j=1}^s (T\times T)^{pn_j}(F_j\otimes \overline{F_j})\, d(\mu\times\mu)=0.
	$$
	
	For $w\in \mathbb{N}$ let  $W$ denote the product of the first $w$ primes.
	Following the proof of \cite[Theorem~1.3]{FHK2} (which uses the Gowers uniformity of the $W$-tricked von Mangoldt function  established in   \cite{GT10b, GT12, GTZ12c}) we get that
	the average on the left hand side  is equal to
	$$
	\lim_{W\to \infty}\mathbb{E}_{(k,W)=1}\, \mathbb{E}_{n\in\mathbb{N}}\, \int \prod_{j=1}^s (T\times T)^{(nW+k)n_j}(F_j\otimes \overline{F_j})\,
	d(\mu\times\mu),
	$$
	where the average
	$\mathbb{E}_{(k,W)=1}$ is taken over those $k\in \{1,\ldots, W-1\}$ such that $(k,W)=1$. In order to show that this limit vanishes, it  suffices to show that for all  distinct  $l_1,\ldots, l_s\in \mathbb{N}$ and for  arbitrary
	$k_1,\ldots, k_s\in \mathbb{N}$ we have
	$$
	\mathbb{E}_{n\in\mathbb{N}}\, \int \prod_{j=1}^s (T\times T)^{l_jn+k_j}(F_j\otimes \overline{F_j})\,
	d(\mu\times\mu)=0.
	$$
	It follows from \cite[Theorem~A.8]{Lei05} (for $s\geq 3$, but a simple argument works for $s=2$) that in order to establish this identity it suffices to show that $\nnorm{F_1\otimes \overline{F_1}}_{s-1,T\times T}=0$. But this follows since   $\nnorm{F_1\otimes \overline{F_1}}_{s-1,T\times T}\leq \nnorm{F_1}_{s,T}^2$ and  by assumption we have $\nnorm{F_1}_{s,T}=0$. This completes the proof.
\end{proof}

\begin{proof}[Proof of Theorem~\ref{T:Tao'}]
	Arguing by contradiction, suppose that  the  conclusion fails. Then there exist  multiplicative functions $f_2,\ldots, f_s\in \mathcal{M}$,  distinct  $n_1,\ldots, n_s\in \mathbb{N}$,
	and a subsequence $(N_k')_{k\in \mathbb{N}}$  of $(N_k)_{k\in\mathbb{N}}$,   such that for   ${\bf I}':=([N_k'])_{k\in\mathbb{N}}$  the limit
	$\mathbb{E}^{\log}_{m\in \mathbf{I}'}\, \prod_{j=1}^s  g_j(l_j m+ k_j)$  exists for  every $s, k_1,\ldots, k_s,l_1,\ldots, l_s\in \mathbb{N}$, and $g_1,\ldots, g_s\in \{a_1,\ldots,a_s, \overline{a_1},\ldots, \overline{a_s} \}$,  and such that
	$$
	\mathbb{E}^{\log}_{m\in \mathbf{I}'}\, \prod_{j=1}^s  f_j(m+ n_j)\neq 0.
	$$
	Using
	Corollary~\ref{C:Tao},  we will get a contradiction if we show that
	\begin{equation}\label{E:iszero}
		\mathbb{E}_{p\in \P}
		\big|\mathbb{E}^{\log}_{m\in {\bf I'}}\, \prod_{j=1}^s f_j(m+pn_j)\big|=0.
	\end{equation}
	In order to prove this identity we will reinterpret it in ergodic terms. Using a variant of Proposition~\ref{P:correspondence}  which applies to several sequences (see \cite[Proposition~3.3]{F16}) we get that  there exist a system
	$(X,\mathcal{X},\mu,T)$ and functions $F_1,\ldots, F_s\in L^\infty(\mu)$ such that
	\begin{equation}\label{E:324}
		\mathbb{E}^{\log}_{m\in {\bf I'}}\, \prod_{j=1}^s f_j(m+pn_j)=\int \prod_{j=1}^s T^{pn_j}F_j\, d\mu
	\end{equation}
	holds for every $p\in \P$  and
	$$
	\norm{f_1}_{U^{s}_{\text{log}}(\mathbf{I}')}= \nnorm{F_1}_{s}.
	$$
	Since by assumption $\norm{f_1}_{U^{s}_{\text{log}}(\mathbf{I})}=0$ and $\mathbf{I}'$ is a subsequence of $\mathbf{I}$,  we have  $\norm{f_1}_{U^{s}_{\text{log}}(\mathbf{I}')}=0$. Hence,   $\nnorm{F_1}_{s}=0$, and   Proposition~\ref{P:ergodic}
	gives that
	$$
	\mathbb{E}_{p\in \P}\Big|\int \prod_{j=1}^s T^{pn_j}F_j\, d\mu\Big|=0.
	$$
	We deduce from this and identity \eqref{E:324} that equation \eqref{E:iszero} holds. This completes the proof.
\end{proof}

\section{Nilmanifolds, nilcharacters, and nilsequences}

\subsection{Nilmanifolds}\label{SS:nilmanifolds} If $G$ is a group  we let
$G_1:=G$ and $G_{j+1}:=[G,G_j]$, $j\in \mathbb{N}$. We say that $G$ is {\em $s$-step nilpotent} if  $G_{s+1}$ is the trivial
group.
An {\em $s$-step nilmanifold} is a homogeneous space $X=G/\Gamma$, where  $G$ is an $s$-step nilpotent Lie group
and $\Gamma$ is a discrete cocompact subgroup of $G$. With $e_X$ we denote  the image in $X$ of the unit element of $G$.
An {\it $s$-step  nilsystem} is
a system of the form $(X, G/\Gamma, m_X, T_b)$,  where $X=G/\Gamma$ is a $k$-step nilmanifold, $b\in G$,  $T_b\colon X\to X$ is defined by  $T_b(g\cdot e_X) \mathrel{\mathop:}= (bg)\cdot e_X$ for  $g\in G$,  $m_X$  is
the normalized Haar measure on $X$, and  $G/\Gamma$  is the completion
of the Borel $\sigma$-algebra of $G/\Gamma$.
We call the map $T_b$ or the element $b$ a {\em nilrotation}. Henceforth,
we assume that
every nilsystem is equipped with a fixed Riemannian metric $d$.
If $\Psi$ is a function on $X$  we let
$
\norm \Psi_{\lip(X)}:=\sup_{x\in X}|\Psi(x)|+\sup_{\substack{x,y\in X\\
		x\neq y}}\frac{|\Psi(x)-\Psi(y)|}{d(x,y)}. $
With $\lip(X)$ we denote the set of all functions $\Psi\colon X\to \mathbb{C}$ with
bounded   $\lip(X)$-norm.

If  $H$ is a closed subgroup of $G$, then  it is shown in \cite[Section~2.2]{Lei05a} that the following three properties are equivalent:
\begin{itemize}
	\item $H\cdot \Gamma$ is closed in $G$;
	
	\item $H\cdot e_X$ is closed in $X$;
	
	\item $H\cap \Gamma$ is cocompact in $H$.
\end{itemize}
For any such $H$, the nilmanifold $H/(H\cap \Gamma)$ is called a {\em sub-nilmanifold} of $X$.

With  $G^0$ we denote the connected component of the identity element in $G$ (this is a normal subgroup
of $G$).
If the nilsystem is ergodic, then since  $G':=\langle G^0,b \rangle$  is  a non-empty open subgroup of
$G$ that is invariant under $b$, we have $G'\cdot e_X=X$. Hence, $X=G'/\Gamma'$ where $\Gamma'=G'\cap \Gamma$.
For every  ergodic
nilsystem we will  use such a representation for $X$ and thus  assume that
$G=\langle G^0,b \rangle.$
This implies (see for example~\cite[Theorem~4.1]{BHK05})  that for $j\geq 2$  all the  commutator subgroups $G_j$ are connected.

Throughout the article, in the case of ergodic nilsystems,  we are going to use these properties without further reference.

\subsection{Equidistribution} \label{SS:equidistribution}
Let  $X=G/\Gamma$ be a nilmanifold. We say that a sequence
$g\colon \mathbb{N}\to X$ is equidistributed in  $X$
if  for every $F\in C(X)$ we have
$$
\lim_{N\to\infty} \mathbb{E}_{n\in [N]} F(g(n))=\int F \ dm_X
$$
where $m_X$ denotes the normalized Haar measure on $X$.

It is proved in \cite{Les91} (see also \cite{Lei05a}) that for every $b\in G$
the set $Y=\overline{\{b^n\cdot e_X, n\in \mathbb{N}\}}$ is a sub-nilmanifold of $X$,
the nilrotation $b$ acts ergodically  in $Y$, and the sequence
$(b^n\cdot y)_{n\in\mathbb{N}}$
is equidistributed in $Y$ for every $y\in Y$. Furthermore,
we can represent $Y$ as $Y=H/\Delta$ where  $H$ is a closed subgroup of $G$ that contains the element $b$ (see the remark  following~\cite[Theorem~2.21]{Lei05a}).
If $Y$ is connected,  then for every $k\in \mathbb{N}$ the nilrotation  $b^k$ acts ergodically on $Y$. If $Y$ is not connected, then there exists $r\in \mathbb{N}$ such
that  the nilrotation $b^r$ acts ergodically on the connected component $Y^0$
of $e_X$ in $Y$.

\subsection{Vertical nilcharacters on $X$ and on $X^0$}\label{SS:nilcharacters}
Let $s\in \mathbb{N}$ and  $X=G/\Gamma$ be a (not necessarily connected) $s$-step nilmanifold  and suppose that   $G=\langle G^0,b \rangle$ for some  $b\in G$.
If $s\geq 2$, then $G_s$ is connected and the group $K_s:=G_s/(G_s\cap\Gamma)$ is a finite dimensional torus (perhaps the trivial one). Let $\widehat{K_s}$ be the dual group of $K_s$; it consists of the characters of $G_s$ that  are $(\Gamma\cap G_s)$-invariant.
A \emph{vertical nilcharacter} of $X$  with \emph{frequency} $\chi$,
where   $\chi\in\widehat{K_s}$,
is a
function $\Phi\in \lip(X)$   that  satisfies
$$
\Phi(u\cdot x)=\chi(u)\, \Phi(x), \ \text{ for every } \ u\in G_s \ \text{  and }  \ x\in X.
$$
If $\chi$ is a non-trivial character of $K_s$,  we say that  $\Phi$
is a \emph{non-trivial vertical nilcharacter},  otherwise we say that it is  a \emph{trivial vertical nilcharacter}.
It is known that the linear span of vertical nilcharacters    is dense in $C(X)$ with the uniform norm (see for example \cite[Proof of Lemma~2.7]{GT12} or \cite[Excercise 1.6.20]{Tao15}).

If the nilmanifold $X$ is not connected, let  $X^0$ be the connected component of $e_X$ in $X$. We claim  that for $s\geq 2$ the restriction of a non-trivial vertical nilcharacter $\Phi$ of $X$ onto $X^0$ is a non-trivial vertical nilcharacter of $X^0$ with  the same frequency.
To see this, note first that
since
$(G^0\Gamma)/\Gamma$
is a non-empty closed and open  subset of the connected space  $X^0$,  we have   $X^0=(G^0\Gamma)/\Gamma$.
It thus suffices to  show that $(G^0\Gamma)_s=G_s$.
To this end, let   $r$ be the smallest integer such that $b^r\in G^0\Gamma$. Then   $G/(G^0\Gamma)$ is isomorphic to the cyclic group $\mathbb{Z}_r$. By induction, for every $k\geq 1$ and all $g_1,\dots,g_k\in G$, we have
$$
[[\dots[g_1,g_2],g_3],\dots,g_k]^{r^k}=[[\dots[g_1^r,g_2^r],g_3^r],\dots,g_k^r]
\bmod G_{k+1}.
$$
Letting $k=s$ and using that $G_{s+1}$ is trivial,  we get for all  $g_1,\dots,g_s\in G$ that
$$[[\dots[g_1,g_2],g_3],\dots,g_s]^{r^s}=
[[\dots[g_1^r,g_2^r],g_3^r],\dots,g_s^r]\in (G^0\Gamma)_s.$$
Using this and because $G_s$ is  Abelian  and  spanned by the elements $[[\dots[g_1,g_2],g_3],\dots,g_s]$, we deduce that for every $h\in G_s$ we have $h^{r^s}\in(G^0\Gamma)_s$. Since $G_s$ is connected for $s\geq 2$  it is  divisible, hence the map $h\mapsto h^{r^s}$ is onto, and we conclude that $(G^0\Gamma)_s=G_s$.

\subsection{Nilsequences} \label{SS:nilsequences}
Following  \cite{BHK05}  we define:
\begin{definition} If $X=G/\Gamma$ is an $s$-step nilmanifold,
	$F\in C(X)$, and $b\in G$, we call the sequence $(F(b^n \cdot e_X))_{n\in \mathbb{N}}$
	an  \emph{$s$-step nilsequence} (we omit the adjective ``basic''). A  \emph{$0$-step nilsequence}
	is a constant sequence.
\end{definition}
\begin{remarks}
	$\bullet$ As remarked in Section~\ref{SS:equidistribution}, the set
	$Y=\overline{\{ b^n\cdot e_X, n\in \mathbb{N}\}}$ is  a sub-nilmanifold  of $X$ that can be represented as  $Y=H/\Delta$ for some closed subgroup $H$ of $G$ with $b\in H$. Thus, upon replacing $X$ with $Y$ we can assume that $b$ is an ergodic nilrotation of $X$.
	
	$\bullet$ For every $x=g\Gamma\in X$, the sequence $(F(b^nx))_{n\in \mathbb{N}}$ is a nilsequence,
	as it can be represented in the form  $(F'(b'^n\cdot e_X))_{n\in \mathbb{N}}$
	where $g':=g^{-1}bg$ and $F'(x):=F(gx)$, $x\in X$.
\end{remarks}

\subsubsection{Nilsequences of bounded complexity}\label{SS:bddcomplexity} To every nilmanifold $X$ (equipped with a Riemannian metric)
we associate a class of nilsequences of ``bounded complexity'' which will be used
in the formulation of the inverse theorem in the next section.
\begin{definition}
	Let $X=G/\Gamma$ be a nilmanifold. We let
	$\Psi_X$  be the set of nilsequences of the form $(\Psi(b^nx))_{n\in\mathbb{N}}$
	where $b\in G$, $x\in X$, and $\Psi\in \lip(X)$ satisfies  $\norm{\Psi}_{\lip(X)}\leq 1$.
\end{definition}
\begin{remark}
	Although $\Psi_X$ is not an algebra,  there exists a nilmanifold $X'$ (take $X'=X\times X$
	with a suitable Riemannian metric) such that $\Psi_{X'}$   contains the sum and the product of any two elements of $\Psi_X$. We will often use this observation without further notice.
\end{remark}

\subsubsection{Approximation by multiple-correlation sequences}
The next lemma will help us establish certain anti-uniformity properties of nilsequences that will be needed later.   It is a consequence of
\cite[Proposition~2.4]{Fr15}.
\begin{lemma}\label{L:nilkey}
	Let $s\in \mathbb{N}$ and $X$ be an $s$-step nilmanifold. Then for every $\varepsilon, L>0$ there exists $M=M(\varepsilon, X,L)$ such that the following holds:
	If $\psi\in L\cdot \Psi_X$,
	then
	there exist a
	system $(Y,\mathcal{Y},\mu,T)$ and functions $F_0,\ldots, F_{s}\in L^\infty(\mu)$,
	all bounded by $M$,
	such  that the
	sequence $(b(n))_{n\in\mathbb{N}}$, defined by
	\begin{equation}\label{E:bn}
		b(n):=\int  F_0 \cdot T^{k_{1} n}F_{1}\cdot  \ldots \cdot  T^{k_{s} n}F_{s}\ d\mu, \quad n\in \mathbb{N},
	\end{equation}
	where
	$k_j:=\frac{(s+1)!j}{j+1}$ for $j=1,\ldots, s$, satisfies
	$$
	\norm{\psi-b}_\infty\leq \varepsilon.
	$$
\end{lemma}
\begin{remark}
	Alternatively, we can use  as approximants sequences of the form
	$b(n):=\lim_{M\to\infty} \mathbb{E}_{m\in [M]} a_0(m)\cdot a_1(m+k_1n)\cdot \ldots\cdot a_s(m+k_sn), n\in \mathbb{N}, $
	where for $j=0,\ldots, s$ the sequences $ a_j\in \ell^\infty(\mathbb{N})$ are defined by $a_j(m):=F_j(T^my_0), m\in \mathbb{N},$ for suitable  $y_0\in Y$.
\end{remark}
\begin{proof}
	Let $\varepsilon>0$.
	First note that since the space of functions on $(X,d_X)$ with  Lipschitz
	constant at most $L$  is compact with respect to the  $\norm{\cdot}_\infty$-norm, we can cover
	this space by a finite number of  $\norm{\cdot}_\infty$-balls of radius at most $\varepsilon$.
	It follows  from this that in order to verify the asserted approximation property,
	it suffices to verify the property for every fixed nilsequence $\psi$ without asking for additional
	uniformity properties for  the $L^\infty(\mu)$ norms of the functions $F_0,\ldots, F_{s}\in L^\infty(\mu)$. This statement now follows immediately from
	\cite[Proposition~2.4]{Fr15}.
\end{proof}

\subsubsection{Reduction of degree of nilpotency} The next result will be used in the proof of the inverse theorem in the next section.   It is a direct consequence of  the constructions in \cite[Section~7]{GT12} and it is
stated in a form equivalent to the one below in \cite[Lemma~1.6.13]{Tao12} (for $\Phi_1=\Phi_2$ but the same argument works in the more general case):
\begin{proposition}[Green, Tao \cite{GT12}]\label{P:GT}
	For  $s\geq 2$ let $X=G/\Gamma$	be an $s$-step nilmanifold. Then there exist an $(s-1)$-step nilmanifold $Y$ and $C=C(X)>0$ such that
	for all vertical nilcharacters
	$\Phi_1, \Phi_2$  of $X$ having the same frequency and satisfying  $\norm{\Phi_j}_{\lip(X)}\leq 1$, $j=1,2$,  every $b\in G$, and every
	$h\in \mathbb{N}$,  the sequence
	$(\Phi_1(b^{n+h}\cdot e_X)\, \overline{\Phi_2(b^n\cdot e_X)})_{n\in\mathbb{N}}$ 	is an $(s-1)$-step nilsequence in $C\cdot \Psi_Y$.
\end{proposition}
To illustrate this result by an example, take
the $2$-step
nilsequence $(\e(n^2\alpha))_{n\in\mathbb{N}}$ which can be defined by a vertical nilcharacter $\Phi$ on the Heisenberg nilmanifold.  We take $\Phi_1=\Phi_2=\Phi$, then   the difference operation results to the $1$-step nilsequences  $(\e(2nh\alpha+h^2\alpha))_{n\in\mathbb{N}}$ which can be represented as
$(\Phi_h(n\beta))_{n\in\mathbb{N}}$, $h\in \mathbb{N}$,  where $\Phi_h\colon \T\to\mathbb{C}$ is defined by $\Phi_h(t):=\e(h^2\alpha)\, \e(t)$, $h\in\mathbb{N}$, and $\beta:=2h\alpha$.

\section{$U^s(\mathbf{I})$-inverse theorem for ergodic sequences}
Henceforth, we assume that $\mathbf{I}=(I_N)_{N\in\mathbb{N}}$ is a  sequence of intervals with $|I_N|\to \infty$.

\subsection{Statement of the inverse theorem}
The goal of this section is to prove the following  inverse theorem:
\begin{theorem}\label{T:ergodicInverse}
	Let $s\in \mathbb{N}$ and suppose that the sequence $a\in \ell^\infty(\mathbb{N})$ is   ergodic for Ces\`aro averages on the sequence of intervals $\mathbf{I}=(I_N)_{N\in\mathbb{N}}$.
	If $\norm{a}_{U^{s}(\mathbf{I})}>0$,  then
	there exist an $(s-1)$-step nilsequence $\phi$  and an  $(s-2)$-step nilmanifold $Y$  such that
	$$
	\limsup_{M\to \infty}\limsup_{N\to\infty} \mathbb{E}_{n\in I_N}\sup_{\psi\in \Psi_Y} |\mathbb{E}_{m\in [n,n+M]}\, a(m)\, \phi(m)\, \psi(m)|>0.
	$$
\end{theorem}
\begin{remarks}
	$\bullet$	A variant for logarithmic averages also holds; one needs to  assume ergodicity for logarithmic averages and replace $\norm{a}_{U^{s+1}(\mathbf{I})}$ with $\norm{a}_{U^{s+1}_\text{log}(\mathbf{I})}$ and
	$\mathbb{E}_{n\in I_N}$ with $\mathbb{E}^{\log}_{n\in I_N}$.

	$\bullet$
	The ergodicity assumption is necessary. Indeed,  if
	$ a(n):=\sum_{k=1}^\infty \e(n \alpha_k)\, \one_{[k^2,(k+1)^2)}(n)$, $n\in \mathbb{N}$, where the sequence of real numbers $(\alpha_k)_{k\in \mathbb{N}}$ is chosen appropriately,   then for every $1$-step nilsequence $\phi$ we have
	$$
	\lim_{M\to \infty}\limsup_{N\to\infty} \mathbb{E}_{n\in [N]}|\mathbb{E}_{m\in [n,n+M]}\, a(m)\, \phi(m)|=0.
	$$
	On the other hand for $\mathbf{I}=([N])_{N\in\mathbb{N}}$ we have
	$\norm{a}_{U^2(\mathbf{I})}=1$.
	
	$\bullet$ For general (not necessarily ergodic) sequences $a\in \ell^\infty(\mathbb{N})$ it can be shown  (as in \cite[Section~4]{T16}) that
	$\norm{a}_{U^{s}(\mathbf{I})}=0$ holds  if and only if for every $(s-1)$-step nilmanifold $Y$ we have
	$$
	\lim_{M\to \infty} \limsup_{N\to \infty} \mathbb{E}_{n\in I_N}\sup_{\psi\in \Psi_Y}|\mathbb{E}_{m\in [n,n+M]}\,  a(m) \cdot \psi(m)|=0.
	$$
	In particular for $s=2$ and $\mathbf{I}=([N])_{N\in \mathbb{N}}$, we get that $\norm{a}_{U^2(\mathbf{I})}=0$ if and only if
	$$
	\lim_{M\to \infty} \limsup_{N\to \infty} \mathbb{E}_{n\in [N]}\sup_{t}|\mathbb{E}_{m\in [n,n+M]} \, a(m) \cdot \e(mt)|=0.
	$$
	Despite its apparent simplicity,  this condition is very hard to verify for particular arithmetic sequences, and it is still unknown for
	the  Liouville and the M\"obius function.
\end{remarks}

\subsection{Sketch of proof for $s=2$ versus  $s>2$} The proof of Theorem~\ref{T:ergodicInverse} is rather simple for $s=2$; we sketch it in order to motivate and explain  some of the maneuvers needed
in the general case. The argument proceeds as follows:
\begin{itemize}	
	\item We first use ergodicity of the sequence $(a(n))_{n\in\mathbb{N}}$ in order to establish the identity
	$$
	\norm{a}_{U^2(\mathbf{I})}^4
	=
	\lim_{H\to \infty}\mathbb{E}_{h\in [H]} |\mathbb{E}_{n\in {\bf I}} \, a(n+h)\cdot \overline{a(n)} |^2.
	$$
	Using this identity and our assumption $\norm{a}_{U^2(\mathbf{I})}>0$,  we deduce that
	\begin{equation}\label{E:positive}
		\lim_{H\to \infty}\mathbb{E}_{h\in [H]}\big( \mathbb{E}_{n\in {\bf I}} \, a(n+h)\cdot \overline{a(n)}\cdot A(h)\big)>0,
	\end{equation}
	where
	$$
	A(h):=\mathbb{E}_{n\in {\bf I}} \, \overline{ a(n+h)}\cdot a(n), \quad h\in \mathbb{N}.
	$$
	This step generalizes straightforwardly  when $s>2$ and gives
	relation  \eqref{E:0+} below.

	\item We can decompose
	the (positive definite) sequence $(A(n))_{n\in\mathbb{N}}$  into a  structured component which is a trigonometric polynomial sequence and an error term   which  is small in uniform density. Hence,  we can assume that \eqref{E:positive} holds when $A(n)=\e(nt)$, $n\in \mathbb{N},$ for some $t\in \R$.   The use of an  infinitary decomposition result is crucial in order to get for $s=2$ an inverse condition that does not involve a supremum in the inner-most average and for $s\geq 3$ an inverse condition that involves a   supremum over $(s-2)$-step (and not $(s-1)$-step) nilsequences of bounded complexity.  The appropriate decomposition result when $s\geq 3$ is Proposition~\ref{P:InfiniteDec} which is proved using deep results from ergodic theory (the main ingredient is Theorem~\ref{T:HK}). Since in this more complicated setup we cannot later on utilize simple identities that  linear exponential sequences satisfy, we take particular care to use as  structured components sequences which have a very convenient (though seemingly  complicated) form.

	\item After interchanging the averages over $h$ and $n$ in \eqref{E:positive}, we take absolute values, and  deduce that
	$$
	\limsup_{H\to \infty} \limsup_{N\to\infty}\mathbb{E}_{n\in  I_N}
	|\mathbb{E}_{h\in [n,n+H]}
	\, a(h)\cdot \e(ht) |>0,
	$$
	which immediately implies the  conclusion of
	Theorem~\ref{T:ergodicInverse} when  $s=2$. This step is harder when $s>2$ and two additional maneuvers are needed  (described in Steps $3$ and $4$ in the proof of  Theorem~\ref{T:ergodicInverse}). One key idea is to  introduce an additional short range average that allows us to replace some unwanted expressions with $(s-2)$ nilsequences. This part  of the argument uses the finitary decomposition result of Proposition~\ref{P:FiniteDec} which is the reason why we get an inverse condition involving a sup over all $(s-2)$-nilsequences of bounded complexity. Another idea needed is to use Proposition~\ref{P:GT} in order to remove an unwanted supremum over a parameter $h\in \mathbb{N}$; not doing so  would cause serious problems later on when we try  to verify the inverse condition for the class of  multiplicative functions we are interested in.
\end{itemize}

We give the details  of the proof of  Theorem~\ref{T:ergodicInverse}  in the next subsections.

\subsection{Uniformity estimates}
We will use the next estimate in the proof of Lemma~\ref{L:estimate'},
which in turn will be used in the proof of Proposition~\ref{P:FiniteDec}.
\begin{lemma}\label{L:estimate}
	Let $s\geq 2$, $M\in \mathbb{N}$,  and  $a_\epsilon \colon \mathbb{Z}_M\to \mathbb{C}$, $\epsilon\in [\![s]\!]^*$.
	Then
	$$
	\mathbb{E}_{m \in \mathbb{Z}_M}\Big|\mathbb{E}_{\underline{h}\in \mathbb{Z}_M^s}\, \prod_{\epsilon\in [\![s]\!]^*}
	a_\epsilon(m+\epsilon\cdot \underline{h})\Big|\leq  \prod_{\epsilon\in [\![s]\!]^*}\norm{a_\epsilon}_{U^s(\mathbb{Z}_M)}.
	$$
\end{lemma}
\begin{proof}
	Notice that the left hand side is equal to
	$$
	\mathbb{E}_{m \in \mathbb{Z}_M}\Big(\mathbb{E}_{\underline{h}\in \mathbb{Z}_M^s}\, \prod_{\epsilon\in [\![s]\!]}
	a_\epsilon(m+\epsilon\cdot \underline{h})\Big),
	$$
	where $a_{\underline{0}}\colon \mathbb{Z}_M\to \mathbb{C}$ is defined by
	$$
	a_{\underline{0}}(m):=e^{-i\phi_m}, \quad \phi_m:=\arg\Big(\mathbb{E}_{\underline{h}\in \mathbb{Z}_M^s}\, \prod_{\epsilon\in [\![s]\!]^*}
	a_\epsilon(m+\epsilon\cdot \underline{h})\Big).
	$$
	Since $\norm{a_{\underline{0}}}_\infty\leq 1$, the claimed estimate  follows from the Gowers-Cauchy-Schwarz inequality \cite[Lemma~3.8]{Gow01}.
\end{proof}

We  use this  lemma in order to deduce a similar estimate for non-periodic sequences:
\begin{lemma}\label{L:estimate'}
	Let $s\geq 2$, $M\in \mathbb{N}$,  and    $a_\epsilon \colon [(s+1)M]\to \mathbb{C}$, $\epsilon\in [\![s]\!]^*$,  be  bounded by $1$. Then
	\begin{equation}\label{E:estimate}
		\mathbb{E}_{m \in [M]}\Big|\mathbb{E}_{\underline{h}\in [M]^s}\, \prod_{\epsilon\in [\![s]\!]^*}
		a_\epsilon(m+\epsilon\cdot \underline{h})\Big|\leq C_s \Big( \min_{\epsilon\in [\![s]\!]^*}\norm{a_\epsilon}_{U^s(\mathbb{Z}_{(s+1)M})}^\frac{1}{s+1}+ \frac{1}{M}\Big)
	\end{equation}
	where $C_s:=(s+1)^{s+1} ((2s)^s+1)$.
\end{lemma}
\begin{proof}
	Let $\tM:=(s+1)M$.
	We first reduce matters to estimating a similar average over $\mathbb{Z}_{\tM}$. Let $\underline{h}=(h_1,\ldots, h_s)$ and notice that the average in \eqref{E:estimate} is bounded by $(s+1)^{s+1}$ times
	\begin{equation}\label{E:estimate'}
		\mathbb{E}_{m \in \mathbb{Z}_{\tM}}\Big|\mathbb{E}_{\underline{h}\in \mathbb{Z}_{\tM}^{s}}\,  \prod_{j=1}^s \one_{[M]}(h_j)\cdot
		\prod_{\epsilon\in [\![s]\!]^*}
		a_\epsilon(m+\epsilon\cdot \underline{h})\Big|
	\end{equation}
	where the sums $m+\epsilon\cdot \underline{h}$ are taken $\! \! \pmod{\tM}$.

	Next, we reduce matters to estimating a similar average that does not contain the indicator functions.
	Let $R$ be an integer that will be specified later and satisfies
	$0<R< M/2$. We define the ``trapezoid function'' $\phi$ on $\mathbb{Z}_{\tM}$ so
	that $\phi(0)=0$, $\phi$ increases linearly from $0$ to $1$ on the
	interval $[0,R]$, $\phi(r)=1$ for $R\leq r\leq M-R$, $\phi$
	decreases linearly from
	$1$ to $0$ on $[M-R,M]$, and $\phi(r)=0$ for $M<r<\tM$.

	After telescoping, we see that the absolute value of the difference between the average \eqref{E:estimate'}  and the average
	\begin{equation}\label{E:estimate''}
		\mathbb{E}_{m \in \mathbb{Z}_{\tM}}\Big|\mathbb{E}_{\underline{h}\in \mathbb{Z}_{\tM}^{s}}\,  \prod_{j=1}^s \phi(h_j)\cdot
		\prod_{\epsilon\in [\![s]\!]^*}
		a_\epsilon(m+\epsilon\cdot \underline{h})\Big|
	\end{equation}
	is bounded by $2s R/ \tM$.
	Moreover, it is classical that
	$$
	\norm{\widehat\phi}_{l^1(\mathbb{Z}_{\tM})}\leq \frac{2M}R\leq \frac{\tM}R
	$$
	and thus \eqref{E:estimate''} is bounded by
	$$
	\frac{\tM^s}{R^s}\,\max_{\xi_1,\ldots, \xi_s \in\mathbb{Z}_\tM}\mathbb{E}_{m \in \mathbb{Z}_{\tM}}\Big|\mathbb{E}_{\underline{h}\in \mathbb{Z}_{\tM}^{s}}\,  \prod_{j=1}^s \e(h_j\xi_j/\tM)\cdot
	\prod_{\epsilon\in [\![s]\!]^*}
	a_\epsilon(m+\epsilon\cdot \underline{h})\Big|.
	$$

	For $j=1,\ldots, s,$  let $\epsilon_j\in [\![s]\!]^*$ be the element  that has $1$  in the $j$-th coordinate and $0$'s elsewhere.
	Upon replacing $a_{\epsilon_j}(n)$ with $a_{\epsilon_j}(n)\, \e(-n\xi_j/\tM)$, $j=1,\ldots, s$,
	and $a_{(1,\ldots,1)}(n)$ with $a_{(1,\ldots,1)}(n)\, \e(n(\xi_1+\cdots+\xi_s)/\tM)$, and leaving all other sequences unchanged,
	the $U^{s}(\mathbb{Z}_\tM)$-norm of all sequences remains unchanged (we use here that $s\geq 2$) and the term $\prod_{j=1}^s \e(h_j\xi_j/\tM)$ disappears. We are thus left with estimating
	the average
	$$
	\mathbb{E}_{m \in \mathbb{Z}_{\tM}}\Big|\mathbb{E}_{\underline{h}\in \mathbb{Z}_{\tM}^{s}}\,
	\prod_{\epsilon\in [\![s]\!]^*}
	a_\epsilon(m+\epsilon\cdot \underline{h})\Big|.
	$$
	Since the sequences $a_\epsilon, \epsilon\in [\![s]\!]^*$, are bounded by $1$, Lemma~\ref{L:estimate}  implies that the last expression is bounded by
	$$ U:= \min_{\epsilon\in [\![s]\!]^*}\norm{a_\epsilon}_{U^s(\mathbb{Z}_{\tM})}.
	$$

	Combining the preceding estimates, we get that  the average in the statement is bounded
	by $(s+1)^{s+1}$ times
	$$
	\frac {2sR}{\tM} +
	\frac{ \tM^s}{R^s} U.
	$$
	Choosing
	$R:= \lfloor U^\frac{1}{s+1} \tM/(4s)\rfloor+1$ (then $R\leq
	M/2$ for $M\geq 5$)
	we get that  the last quantity is bounded by
	$$
	((4s)^s+1)\, U^\frac{1}{s+1}+\frac{2}{M}
	$$
	when $M\geq 5$. When $M\leq 4$ the asserted estimate is trivial,  completing the proof.
\end{proof}

\subsection{Two decompositions}
We will use the following infinitary decomposition result which is proved using  tools from ergodic theory.
\begin{proposition}\label{P:InfiniteDec}
	Let $s\in \mathbb{N}$ and suppose that the sequence $a\in \ell^\infty(\mathbb{N})$ is   ergodic for Ces\`aro averages  on the sequence of intervals $\mathbf{I}=(I_N)_{N\in\mathbb{N}}$.  Then for every $s\in \mathbb{N}$  the sequence $A\colon\mathbb{N}^s\to\mathbb{C}$ defined by
	$$
	A(\underline{h}):=\mathbb{E}_{n \in\mathbf{I}}\prod_{\epsilon\in [\![s]\!]}\mathcal{C}^{|\epsilon|} a(n+\epsilon\cdot \underline{h}), \quad \underline{h}\in \mathbb{N}^s,
	$$
	admits a decomposition of the form $A=A_\st+A_{\er}$ such that
	\begin{enumerate}
		\item
		$A_\st\colon \mathbb{N}^s\to \mathbb{C}$ is a uniform limit of sequences of the form
		$$
		\underline{h} \mapsto \int \, A_{\text{st},x}(\underline{h}) \, d\mu(x), \quad \underline{h}\in \mathbb{N}^s,
		$$
		where the integration takes place on a probability space $(X,\mathcal{X},\mu)$,  and for $x\in X$ the sequence $A_{\text{st},x}\colon \mathbb{N}^s\to \mathbb{C}$ is defined by
		$$
		A_{\text{st},x}(\underline{h}):=\mathbb{E}_{n \in\mathbf{I}}\prod_{\epsilon\in [\![s]\!]}\mathcal{C}^{|\epsilon|} \phi_x(n+\epsilon\cdot \underline{h}), \quad \underline{h}\in \mathbb{N}^s,
		$$
		where the sequence $\phi_x\colon \mathbb{N}\to \mathbb{C}$
		is an $s$-step  nilsequence with $\norm{\phi_x}_\infty\leq \norm{a}_\infty$, and for $n\in \mathbb{N}$ the map
		$x\mapsto \phi_x(n)$ is $\mu$-measurable; and
		
		\item
		$\lim_{M\to\infty}\mathbb{E}_{\underline{h} \in [M]^s}|A_\er(\underline{h})|=0$.
	\end{enumerate}
\end{proposition}
\begin{remarks}
	$\bullet$ The ergodicity assumption  in this statement is a convenience; we can prove a similar statement without it by using the decomposition result in \cite[Proposition~3.1]{CFH11} in place of Theorem~\ref{T:HK}.
	
	$\bullet$ It can be shown that the sequence $A_\st$ is  a uniform limit of $s$-step nilsequences in $s$ variables, but such a decomposition result is  less useful for our purposes.	
\end{remarks}
\begin{proof}
	Let $(X,\mathcal{X},\mu,T)$ be the ergodic system and $F\in L^\infty(\mu)$ be the function associated to  $(a(n))_{n\in\mathbb{N}}$ and  $\mathbf{I}$ by the  correspondence principle of Proposition~\ref{P:correspondence}.
	Then
	$$
	A(\underline{h}):=\int \prod_{\epsilon\in [\![s]\!]}\mathcal{C}^{|\epsilon|} T^{\epsilon\cdot \underline{h}}F\, d\mu, \quad \underline{h}\in \mathbb{N}^s.
	$$
	We set
	$$
	A_{\text{st}}(\underline{h}):=\int \prod_{\epsilon\in [\![s]\!]}\mathcal{C}^{|\epsilon|} T^{\epsilon\cdot \underline{h}}F_{\text{st}}\, d\mu, \quad \underline{h}\in \mathbb{N}^s,
	$$
	where $F_{\text{st}}:=\mathbb{E}(F|\mathcal{Z}_{s})$ is the orthogonal projection of $F$ onto $L^2(\mathcal{Z}_{s})$ and $\mathcal{Z}_{s}$ is the $\sigma$-algebra defined  in Section~\ref{SS:GHKseminorms}. Furthermore, we let
	$$
	A_{\text{er}}:=A-A_{\text{st}}.
	$$

	We  first deal with  the sequence  $A_\er$. It follows  from  \cite[Theorem~13.1]{HK05a}  that if $F_\epsilon\in L^\infty(\mu)$ for all   $\epsilon\in  [\![s]\!]$ and  $\nnorm{F_\epsilon}_{s+1,T}=0$ for some $\epsilon\in  [\![s]\!]$, then
	$$
	\lim_{M\to \infty}\mathbb{E}_{\underline{h}\in [M]^s}\Big| \int \prod_{\epsilon\in [\![s]\!]} T^{\epsilon\cdot \underline{h}}F_{\epsilon}\, d\mu\Big|=0.
	$$
	Using this, telescoping, and since $F-F_{\text{st}}\, \bot \, L^2(\mathcal{Z}_{s})$ implies $\nnorm{F-F_{\text{st}}}_{s+1,T}=0$, we deduce that
	$$
	\lim_{M\to\infty}\mathbb{E}_{\underline{h}\in [M]^s}|A_\er(\underline{h})|=0.
	$$
	
	Next we establish the asserted  structural property  of the sequence $A_\st$.  Theorem~\ref{T:HK}  gives that   the system $(X,\mathcal{Z}_{s},\mu,T)$ is  an inverse limit of $s$-step nilsystems. It follows from this that the sequence $A_\st$ is a uniform limit of sequences of the form
	$$
	\underline{h}\mapsto
	\int \prod_{\epsilon\in [\![s]\!]}\mathcal{C}^{|\epsilon|}   \Phi(b^{\epsilon\cdot \underline{h}}x)\, d m_X, \quad \underline{h}\in \mathbb{N}^s,
	$$
	where $X=G/\Gamma$ is an $s$-step nilmanifold, $b\in G$, $m_X$ is the normalized Haar measure of $X$,  and $\Phi\in C(X)$ satisfies
	$$
	\norm{\Phi}_\infty\leq \norm{F_{st}}_\infty\leq \norm{F}_\infty=\norm{a}_\infty.
	$$
	As remarked in Section~\ref{SS:nilmanifolds}, for every $\underline{h}\in \mathbb{N}^s$ the limit $\mathbb{E}_{n\in\mathbf{I}}\,  \prod_{\epsilon\in [\![s]\!]}\mathcal{C}^{|\epsilon|}  \Phi(b^{n+\epsilon\cdot \underline{h}}x)$
	exists. Using this property, the bounded convergence theorem, and
	the preservation of $m_X$ by left translation by $b^n, n\in\mathbb{N}$, we get 	
		$$
		\int \prod_{\epsilon\in [\![s]\!]}\mathcal{C}^{|\epsilon|}  \Phi(b^{\epsilon\cdot \underline{h}}x)\, d m_X
		=\int \mathbb{E}_{n\in\mathbf{I}}\,  \prod_{\epsilon\in [\![s]\!]}\mathcal{C}^{|\epsilon|}  \Phi(b^{n+\epsilon\cdot \underline{h}}x)\, d m_X
		=\int \mathbb{E}_{n\in \mathbf{I}}\,   \prod_{\epsilon\in [\![s]\!]}\mathcal{C}^{|\epsilon|}  \phi_x(n+\epsilon\cdot \underline{h})\, d m_X,
	$$
	where for $x\in  X$ we let $\phi_x(n):=\Phi(b^nx)$, $n\in\mathbb{N}$.
	Note that  $(\phi_x(n))_{n\in\mathbb{N}}$ is an  $s$-step nilsequence for every $x\in X$,
	$\norm{\phi_x}_\infty\leq \norm{\Phi}_\infty\leq \norm{a}_\infty$,
	and  for fixed $n\in\mathbb{N}$ the map $x\mapsto \phi_x(n)=\Phi(b^nx)$ is $m_X$-measurable. This completes the proof.
\end{proof}

The next result is proved in \cite{GT10} using the finitary inverse theorem for the Gowers uniformity norms in \cite{GTZ12c}.
Recall that for a given nilmanifold $X$ the set $\Psi_X$ was defined in Section~\ref{SS:bddcomplexity}.
\begin{theorem}[Green, Tao \mbox{\cite[Proposition~2.7]{GT10}}]\label{GT10}
	Let $\delta>0$ and $s\geq 2$ be an integer. Then there exist a positive number $L:=L(\delta,s)$ and an  $(s-1)$-step nilmanifold $X=X(\delta,s)$ such that the following holds: For every large enough  $M\in \mathbb{N}$,
	every   $a\colon [M]\to \mathbb{C}$ that  is bounded by $1$, admits a
	decomposition $a=a_{M,\st}+a_{M,\un}$ such that
	\begin{enumerate}
		\item
		$a_{M,\st}$ is an $(s-1)$-step nilsequence  in $L\cdot \Psi_X$ and $\norm{a_{N,\st}}_\infty\leq 4$; and
		\item
		$\norm{a_{M,\text{un}}}_{U^s(\mathbb{Z}_M)}\leq \delta$.
	\end{enumerate}
	\begin{remark}
		The proof in \cite{GT10} is given for sequences with values in $[0,1]$, the stated result
		follows since any sequence with values
		on the unit disc is a complex linear combination of four  sequences with values in $[0,1]$. Moreover, the result  in \cite{GT10} is given for the $U^s[M]$ norms, the stated version follows since
		for large enough $M\in \mathbb{N}$ the  $U^s(\mathbb{Z}_{M})$ and $U^s[M]$ norms are
		comparable (see \cite[Lemma~A.4]{FH15a}).
		Lastly,   the statement in \cite{GT10} contains a third term that is  small in  $L^2[M]$,
		this term has been absorbed in the $a_{M,\text{un}}$ term in our statement.
	\end{remark}
\end{theorem}
We will combine  Lemma~\ref{L:estimate'} and Theorem~\ref{GT10}  in order to establish the following finitary decomposition result:
\begin{proposition}\label{P:FiniteDec}
	For every $\varepsilon>0$ and $s\geq 2$  there exist $C=C(\varepsilon,s)>0$ and an  $(s-1)$-step  nilmanifold $Y=Y(\varepsilon, s)$ such that the following holds: For every large enough  $M\in \mathbb{N}$ and  every
	$a\colon [(s+1)M]\to \mathbb{C}$ that is bounded by $1$,
	the sequence $A\colon [M]\to\mathbb{C}$ defined by
	$$
	A(m):=\mathbb{E}_{\underline{h}\in [M]^s}\, \prod_{\epsilon\in [\![s]\!]^*}\mathcal{C}^{|\epsilon|}
	a(m+\epsilon\cdot \underline{h}), \quad m\in [M],
	$$
	admits a decomposition of the form $A=A_{M,\st}+A_{M,\er}$ such that
	\begin{enumerate}
		\item
		$
		A_{M,\st}(m):=\mathbb{E}_{\underline{h}\in [M]^s}\,  \psi_{M,\underline{h}}(m),$ $m\in [M],$
		where 	$\psi_{M,\underline{h}}\in C\cdot \Psi_Y$ for all  $\underline{h} \in [M]^s$; and
		\item
		$
		\mathbb{E}_{m \in [M]}|A_{M,\er}(m)|\leq\varepsilon$.
	\end{enumerate}
\end{proposition}
\begin{proof}
	Let $\varepsilon>0$ and $s\geq 2$.
	We use  the decomposition result of Theorem~\ref{GT10} for  $\delta=\delta(\varepsilon,s)$ to be determined momentarily.  We get an $(s-1)$-step nilmanifold $X=X(\delta,s)$ an $L=L(\delta,s)>0$,
	such that for  every large enough $M\in \mathbb{N}$ we have a decomposition $a(n)=a_{M,\st}(n)+a_{M,\un}(n)$, $n\in [(s+1)M]$, where $a_{M,\st}\in L\cdot \Psi_X$,   $\norm{a_{N,\st}}_\infty\leq 4$,  and 	$\norm{a_{M,\text{un}}}_{U^s(\mathbb{Z}_{(s+1)M})}\leq \delta$.
	
	We let
	$$
	A_{M,\st}(m):= \mathbb{E}_{\underline{h}\in [M]^s}\, \prod_{\epsilon\in [\![s]\!]^*}\mathcal{C}^{|\epsilon|}
	\psi_M(m+\epsilon\cdot \underline{h}),  \quad m\in [M],
	$$
	where 	 $\psi_M:=a_{M,\text{st}}$, and
	$$
	A_{M,\er}:=A-A_{M,\st}.
	$$

	Using Lemma~\ref{L:estimate'} and telescoping, we deduce that if $\delta$ is sufficiently small, depending on $\varepsilon$ and $s$ only, then  $A_{M,\er}$ satisfies $(ii)$.
	
	It remains to deal with the term $A_{M,\st}$.
	Since $\psi_M\in L\cdot \Psi_X$, there exists an $(s-1)$-step nilmanifold $Y=Y(\varepsilon,s)$ (take $Y:=X^{2^s-1}$ with the product metric) and a constant $C=C(\delta,s)>0$ such that for every $M\in\mathbb{N}$ and $\underline{h}\in [M]^s$  the sequence $\psi_{M,\underline{h}}\colon \mathbb{N}\to \mathbb{C}$, defined by $\psi_{M,\underline{h}}(m):=\prod_{\epsilon\in [\![s]\!]^*}\mathcal{C}^{|\epsilon|}
	\psi_M(m+\epsilon\cdot \underline{h})$, $m\in\mathbb{N}$, is in $C\cdot \Psi_Y$.
	This completes the proof.
\end{proof}

\subsection{Proof of Theorem~\ref{T:ergodicInverse} }
Let $s\geq 2$.
Without loss of generality we can assume that $\norm{a}_\infty\leq 1$.

\medskip

\noindent {\bf Step 1 (Using ergodicity).}  Using the ergodicity of $a(n)$for Ces\`aro averages on $\mathbf{I}$ (this is the only place where we make essential use of ergodicity in the proof of Theorem~\ref{T:ergodicInverse}) we get that
$$
\norm{a}_{U^s(\mathbf{I})}^{2^s}=\mathbb{E}_{\underline{h}\in \mathbb{N}^{s-1}}\Big|\mathbb{E}_{n\in\mathbf{I}}\prod_{\epsilon\in [\![s-1]\!]}\mathcal{C}^{|\epsilon|} a(n+\epsilon\cdot \underline{h})\Big|^2.
$$
To see this note that
\begin{align*}
	\norm{a}_{U^s(\mathbf{I})}^{2^s}=&\mathbb{E}_{\underline{h}\in \mathbb{N}^s}\Big(\mathbb{E}_{n\in\mathbf{I}}\prod_{\epsilon\in [\![s]\!]}\mathcal{C}^{|\epsilon|} a(n+\epsilon\cdot \underline{h})\Big)
	\\
	=&\mathbb{E}_{\underline{h}\in \mathbb{N}^{s-1}}\mathbb{E}_{h\in \mathbb{N}} \Big(\mathbb{E}_{n\in\mathbf{I}}\prod_{\epsilon\in [\![s-1]\!]}\mathcal{C}^{|\epsilon|}a(n+\epsilon\cdot \underline{h}+h)\cdot
	\prod_{\epsilon\in [\![s-1]\!]}\mathcal{C}^{|\epsilon|}\overline{a(n+\epsilon\cdot \underline{h})}\Big) \\
	=&\mathbb{E}_{\underline{h}\in \mathbb{N}^{s-1}} \Big(\mathbb{E}_{n\in\mathbf{I}}\prod_{\epsilon\in [\![s-1]\!]}\mathcal{C}^{|\epsilon|}a(n+\epsilon\cdot \underline{h}) \cdot
	\mathbb{E}_{n\in\mathbf{I}}\prod_{\epsilon\in [\![s-1]\!]}\mathcal{C}^{|\epsilon|}\overline{a(n+\epsilon\cdot \underline{h})} \Big) \\
	=&\mathbb{E}_{\underline{h}\in \mathbb{N}^{s-1}}\Big|\mathbb{E}_{n\in\mathbf{I}}\prod_{\epsilon\in [\![s-1]\!]}\mathcal{C}^{|\epsilon|} a(n+\epsilon\cdot \underline{h})\Big|^2,
\end{align*}
where the first identity follows from \eqref{E:Us}, the second follows from the remarks made in Section~\ref{SS:uniformity},  and the third from our  ergodicity assumption using identity \eqref{E:ET}.

\medskip

\noindent {\bf Step 2 (Using an  infinitary decomposition).}
Hence, if  $\norm{a}_{U^s(\mathbf{I})}>0$, then
\begin{equation}\label{E:0+}
	\mathbb{E}_{\underline{h}\in \mathbb{N}^{s-1}}\Big(\mathbb{E}_{n\in\mathbf{I}}\prod_{\epsilon\in [\![s-1]\!]}\mathcal{C}^{|\epsilon|} a(n+\epsilon\cdot \underline{h})\cdot A(\underline{h})\Big)>0,
\end{equation}
where
$$
A(\underline{h}):=\mathbb{E}_{n'\in\mathbf{I}}\prod_{\epsilon\in [\![s-1]\!]}\mathcal{C}^{|\epsilon|+1} a(n'+\epsilon\cdot \underline{h}).
$$
As remarked in Section~\ref{SS:uniformity}, we can replace the average $\mathbb{E}_{\underline{h}\in \mathbb{N}^{s-1}}$ with  $\lim_{H\to\infty}\mathbb{E}_{\underline{h}\in [H]^{s-1}}$.
By Proposition~\ref{P:InfiniteDec} we have a decomposition
$$
A(\underline{h})=A_{\text{st}}(\underline{h})+A_{\text{er}}(\underline{h}), \quad \underline{h}\in \mathbb{N}^{s-1},
$$
where
\begin{enumerate}
	\item
	$A_{\text{st}}(\underline{h})$ is  a uniform limit of sequences of the form
	$$
	\underline{h} \mapsto
	\int \, A_{\text{st},x}(\underline{h}) \, d\mu(x), \quad \underline{h}\in \mathbb{N}^{s-1},
	$$
	such that
	$$A_{\text{st},x}(\underline{h}):=
	\mathbb{E}_{n'\in\mathbf{I}}\prod_{\epsilon\in [\![s-1]\!]}\mathcal{C}^{|\epsilon|+1} \phi_x(n'+\epsilon\cdot \underline{h}), \quad \underline{h}\in \mathbb{N}^{s-1},
	$$
	where  for $x\in X$ the sequence  $\phi_x\colon \mathbb{N}\to \mathbb{C}$
	is a basic $(s-1)$-step nilsequence with $\norm{\phi_x}_\infty\leq \norm{a}_\infty\leq 1$,
	and for $n\in\mathbb{N}$ the map $x\mapsto \phi_x(n)$ is $\mu$-measurable; and
	
	\item
	$
	\lim_{M\to\infty}\mathbb{E}_{\underline{h}\in [M]^{s-1}}|A_{\text{er}}(\underline{h})|=0.
	$
\end{enumerate}
Using uniform approximation and the second condition we deduce that
$$
\lim_{M\to \infty} \mathbb{E}_{\underline{h}\in [M]^{s-1}}\Big(\mathbb{E}_{n\in\mathbf{I}}\prod_{\epsilon\in [\![s-1]\!]}\mathcal{C}^{|\epsilon|} a(n+\epsilon\cdot \underline{h})\cdot \int \, A_{\text{st},x}(\underline{h})\, d\mu(x) \Big) >0.
$$
Hence,
$$
\lim_{M\to \infty}\int \, \mathbb{E}_{\underline{h}\in [M]^{s-1}}\Big(\mathbb{E}_{n\in\mathbf{I}}\prod_{\epsilon\in [\![s-1]\!]}\mathcal{C}^{|\epsilon|} a(n+\epsilon\cdot \underline{h})\cdot A_{\text{st},x}(\underline{h})\Big)\, d\mu(x) >0.
$$
Using Fatou's lemma we deduce that  there exists an $x\in X$ such that
$$
\limsup_{M\to \infty} \Big|\mathbb{E}_{\underline{h}\in [M]^{s-1}}\Big(\mathbb{E}_{n\in\mathbf{I}}\prod_{\epsilon\in [\![s-1]\!]}\mathcal{C}^{|\epsilon|} a(n+\epsilon\cdot \underline{h})\cdot A_{\text{st},x}(\underline{h})\Big)\Big| >0.
$$
Using the form of $A_{\text{st},x}$ and the fact that both limits
$\mathbb{E}_{n\in \mathbf{I}}\cdots $ and $\mathbb{E}_{n'\in \mathbf{I}}\cdots $ exist, we get that
$$
\limsup_{M\to \infty}\lim_{N\to \infty}\Big| \mathbb{E}_{n,n'\in I_N} \, a(n)\, \phi_x(n') \, \Big(\mathbb{E}_{\underline{h}\in [M]^{s-1}}\, \prod_{\epsilon\in [\![s-1]\!]^*}\mathcal{C}^{|\epsilon|}
\big(a(n+\epsilon\cdot \underline{h}) \cdot \overline{\phi_x(n'+\epsilon\cdot \underline{h})}\big)\Big)\Big|>0.
$$
Hence, renaming $\phi_x$ as $\phi$, we get that  there exists an $(s-1)$-step nilsequence $\phi$ such that
$$
\limsup_{M\to \infty}\lim_{N\to \infty}\Big| \mathbb{E}_{n,n'\in I_N} \, a(n)\, \phi(n') \, \Big(\mathbb{E}_{\underline{h}\in [M]^{s-1}}\, \prod_{\epsilon\in [\![s-1]\!]^*}\mathcal{C}^{|\epsilon|}
\tilde{a}_{n,n'}(\epsilon\cdot \underline{h})\Big)\Big|>0,
$$
where for $n,n'\in \mathbb{N}$ the sequence $(\tilde{a}_{n,n'}(k))_{k\in\mathbb{N}}$ is defined by
$$
\tilde{a}_{n,n'}(k):=a(n+k)\, \phi(n'+k), \quad k\in \mathbb{N}.
$$

\medskip

\noindent {\bf Step 3 (Using a finitary decomposition).}
Next, we shift the averages over $n$  and $n'$ by $m\in \mathbb{N}$ and average over $m\in [M]$.\footnote{We perform this maneuver in order to  introduce a ``small range parameter'' $m\in [M]$ which will give rise to sequences in $m\in [M]$ for which  finitary decomposition results are applicable.}
We deduce that
$$
\limsup_{M\to \infty}\limsup_{N\to \infty}\mathbb{E}_{n,n'\in I_N} \big| \mathbb{E}_{m\in [M]}  a(m+n)\, \phi(m+n') \, A_{M,n,n'}(m)\big|=\delta>0,
$$
where for $M,n,n'\in\mathbb{N}$ we let
$$
A_{M,n,n'}(m):=\mathbb{E}_{\underline{h}\in [M]^{s-1}}\, \prod_{\epsilon\in [\![s-1]\!]^*}\mathcal{C}^{|\epsilon|}
\tilde{a}_{n,n'}(m+\epsilon\cdot \underline{h}), \quad m\in [M].
$$

For $M,n,n'\in\mathbb{N}$, we use  Proposition~\ref{P:FiniteDec}   for  $\varepsilon:=\delta/3$
in order to decompose the finite sequence $A_{M,n,n'}(m),$ $m\in [M]$.  We get  that there exist $C=C(\delta,s)>0$,  an $(s-2)$-step  nilmanifold $Y=Y(\delta,s)$,
and for large enough  $M\in \mathbb{N}$ there exist
$(s-2)$-step nilsequences $\psi_{M,n,n',\underline{h}}\in C\cdot \Psi_Y$, where  $\underline{h}\in [M]^{s-1}$, $n,n'\in\mathbb{N}$,   such that
$$
\limsup_{M\to\infty}\limsup_{N\to\infty}\mathbb{E}_{n,n'\in I_N} \big| \mathbb{E}_{m\in [M]}  a(m+n)\, \phi(m+n') \, \mathbb{E}_{\underline{h}\in [M]^{s-1}}\, \psi_{M,n,n',\underline{h}}(m)\big|>\delta/2>0.
$$
Hence,
$$
\limsup_{M\to\infty}\limsup_{N\to\infty}\mathbb{E}_{n,n'\in I_N}\mathbb{E}_{\underline{h}\in [M]^{s-1}} \big| \mathbb{E}_{m\in [M]}  a(m+n)\, \phi(m+n') \, \psi_{M,n,n',\underline{h}}(m)\big|>0,
$$
from which we deduce that
$$
\limsup_{M\to\infty}
\limsup_{N\to\infty} \mathbb{E}_{n,n'\in I_N} \sup_{\psi\in \Psi_{Y }}|\mathbb{E}_{m\in [M]}\, a(m+n)\, \phi(m+n') \,\psi(m)|>0.
$$
This implies that (notice that $n\mapsto \psi(n+k)$ is in  $ \Psi_Y$ for every $k\in\mathbb{Z}$)
\begin{equation}\label{E:h}
	\limsup_{M\to\infty} \limsup_{N\to\infty}\mathbb{E}_{n\in I_N} \sup_{\psi\in \Psi_Y, h\in \mathbb{N}}|\mathbb{E}_{m\in [n, n+M]}\, a(m)\, \phi(m+h) \,\psi(m)|>0
\end{equation}
for some $(s-1)$-step nilsequence $\phi$.

\medskip

\noindent {\bf Step 4 (Removing the sup over $h$).}
It remains to show that the supremum over $h\in \mathbb{N}$ can be removed.
As we mentioned in Section~\ref{SS:nilcharacters},
the linear span of vertical nilcharacters is dense in $C(X)$. Hence,   we can assume that
$\phi(n)=\Phi(b^n\cdot e_X), n\in \mathbb{N},$  for some $(s-1)$-step nilmanifold $X$,   $b\in G$,  and vertical nilcharacter $\Phi$ of $X$ with  $\norm{\Phi}_{\lip(X)}\leq 1$. Although we cannot  assume that $|\Phi(x)|=1$ for every $x\in X$,  it is known (see  the proof of  \cite[Proposition~5.6]{TT17}  or  \cite[Lemma~6.4]{GTZ12c}) that there exist $k\in \mathbb{N}$ and vertical nilcharacters $\Phi_1,\ldots, \Phi_k$ of $X$,
which all have the same frequency as $\Phi$,   and satisfy  $\norm{\Phi_j}_{\lip(X)}\leq 1$ for  $j=1,\ldots, k$ and  $|\Phi_1(x)|^2+\cdots+|\Phi_k(x)|^2=1$ for every $x\in X$. For $j=1,\ldots, k$, we let $\phi_j(n):=\Phi_j(b^n\cdot e_X), n\in \mathbb{N}$.
Then $|\phi_1(n)|^2+\cdots +|\phi_k(n)|^2=1$ for every $n\in \mathbb{N}$ and we deduce from this and  \eqref{E:h} that for some $j_0\in \{1,\ldots, k\}$ we have
\begin{equation}\label{E:h'}
	\limsup_{M\to\infty} \limsup_{N\to\infty}\mathbb{E}_{n\in I_N} \sup_{\psi\in \Psi_Y, h\in \mathbb{N}}|\mathbb{E}_{m\in [n, n+M]}\, a(m)\, \phi(m+h) \, |\phi_{j_0}(m)|^2\, \psi(m)|>0.
\end{equation}
It follows from Proposition~\ref{P:GT}  that there exist an $(s-2)$-nilmanifold $\tilde{Y}$ and $\tilde{C}=\tilde{C}(Y)=\tilde{C}(\delta,s)>0$ such that for every $h\in \mathbb{N}$ the sequence $(\phi(n+h)\, \overline{\phi_{j_0}(n)})_{n\in\mathbb{N}}$ is an $(s-2)$-step nilsequence in $\tilde{C}\cdot \Psi_{\tilde{Y}}$.
Thus,  upon
enlarging  the $(s-2)$-nilmanifold $Y$ in     \eqref{E:h'}  the term $\phi(m+h)\, \overline{\phi_{j_0}(m)}$
can be absorbed in the supremum over $\psi\in \Psi_Y$. We deduce that
$$
\limsup_{M\to\infty} \limsup_{N\to\infty}\mathbb{E}_{n\in I_N} \sup_{\psi\in \Psi_Y}|\mathbb{E}_{m\in [n, n+M]}\, a(m)\, \phi_{j_0}(m) \,\psi(m)|>0.
$$
This completes the proof.

\section{$U^s(\mathbf{I})$-uniformity for the Liouville function}
Our goal in this section is to prove Theorem~\ref{T:uniformity}.
Note that this uniformity result combined with Theorem~\ref{T:Tao}
gives Theorem~\ref{T:ErgodicChowla} and combined with Theorem~\ref{T:Tao'} gives Theorem~\ref{T:ErgodicElliott}.
We present the proof of
Theorem~\ref{T:uniformity} for Ces\`aro averages but  a similar argument  also works for logarithmic averages;  wherever needed we  indicate  which statements need to be modified for this purpose.

\subsection{Sketch of the proof} \label{SS:sketch} We proceed by induction as follows:
\begin{itemize}
	\item  For $s=2$ we get that Theorem~\ref{T:uniformity} follows from the ergodicity of $f$ and  Proposition~\ref{P:stronaperiodic} (the strong aperiodicity of $f$ is only used here). Assuming that $s\geq 2$ and  $\norm{f}_{U^s(\mathbf{I})}=0$,  our goal then becomes  to show that $\norm{f}_{U^{s+1}(\mathbf{I})}=0$.
	
	\item
	We first use the inverse result of  Theorem~\ref{T:ergodicInverse} in order  to reduce matters to an orthogonality property of $f$ with
	nilsequences on typical short intervals (see Proposition~\ref{P:Discorrelation'}). Essential use of ergodicity of $f$ is made here.
	
	\item
	The orthogonality property involves a fixed $s$-step nilsequence $\phi$ and a supremum over a set of $(s-1)$ step nilsequences of bounded complexity. If $\phi$ is an $(s-1)$-step nilsequence, then  we are done by the induction hypothesis (the ergodicity of $f$ is used again here)
	via  elementary estimates.
	If not, we reduce matters to the case where the $s$-step  nilsequence  $\phi$ is defined by a non-trivial vertical nilcharacter (see Proposition~\ref{P:Discorrelation''}).

	\item We then 	
	use the orthogonality criterion of Lemma~\ref{L:katai} (the multiplicativity of $f$ is only used here)
	 	  to reduce matters to a purely dynamical statement about ``irrational nilsequences'' (see Proposition~\ref{P:orthogonality}).
	
	\item
	Lastly, we verify the dynamical statement using elementary estimates, qualitative equidistribution results on nilmanifolds, and ideas motivated from \cite{FH15a}.
\end{itemize}

\subsection{Step 1 (Setting up the induction and cases $s=1,2$)}
We prove Theorem~\ref{T:uniformity} by induction on $s\in \mathbb{N}$.
We cover the cases $s=1$ and $s=2$ separately, partly because we want to show their relation to  recently established results, but also  because the inductive step $s\mapsto s+1$ is slightly different when $s\geq 2$.

For $s=1$ we have (recall that $\mathbf{I}=([N_k])_{k\in\mathbb{N}}$)
$$
\norm{f}_{U^1(\mathbf{I})}^2=\lim_{H\to\infty}\mathbb{E}_{h\in [H]}
\big(\mathbb{E}_{n\in \mathbf{I}}\, f(n+h)\, \overline{f(n)})\leq
\limsup_{H\to\infty}\limsup_{N\to\infty}\mathbb{E}_{n\in [N]}|\mathbb{E}_{h\in [H]} f(n+h)|
$$
and the last limit is $0$ by  \cite[Theorem A.1]{MRT15}. Note that this argument did not use our ergodicity assumption on   $f$.
Assuming ergodicity of $f$ for Ces\`aro averages on ${\bf I}$,  then one simply notes that
$\norm{f}_{U^1(\mathbf{I})}=|\mathbb{E}_{n\in \mathbf{I}} f(n)|=0$.

For $s=2$, using our hypothesis that the sequence $(f(n))_{n\in\mathbb{N}}$
is ergodic for Ces\`aro averages on $\mathbf{I}$ we derive exactly as in the first step of the proof of Theorem~\ref{T:ergodicInverse} the identity
$$
\norm{f}_{U^2(\mathbf{I})}^4
=
\lim_{H\to \infty}\mathbb{E}_{h\in [H]} |\mathbb{E}_{n\in {\bf I}} \, f(n+h)\cdot \overline{f(n)} |^2.
$$
This limit is  $0$ by  Proposition~\ref{P:stronaperiodic}.

Suppose now that Theorem~\ref{T:uniformity} holds for $s\geq 2$; in the remaining subsections  we will show that it holds for $s+1$.

\subsection{Step 2 (Using the inverse theorem)} We start by using the inverse theorem proved in the previous section.
It follows from Theorem~\ref{T:ergodicInverse} and Proposition~\ref{E:normequiv}
that in order to prove  Theorem~\ref{T:uniformity}
it suffices to establish the following result:
\begin{proposition}\label{P:Discorrelation'}
	Let $s \geq 2$ and $f\in \mathcal{M}$ be a strongly aperiodic   multiplicative function which is ergodic for Ces\`aro averages on the sequence of intervals
	$\mathbf{I}=([N_k])_{k\in\mathbb{N}}$.
	Then
	for every   $s$-step nilsequence $\phi$  and every $(s-1)$-step nilmanifold $Y$, we have
	$$
	\lim_{M\to \infty}\limsup_{k\to\infty} \mathbb{E}_{n\in [N_k]}\sup_{\psi\in \Psi_Y} |\mathbb{E}_{m\in [n,n+M]}\, f(m)\, \phi(m) \, \psi(m)|=0.
	$$
\end{proposition}
\begin{remark}
	$\bullet$ A variant for logarithmic averages also holds where one  assumes ergodicity for logarithmic averages on $\mathbf{I}$ and replaces $\mathbb{E}_{n\in [N_k]}$ with $\mathbb{E}^{\log}_{n\in [N_k]}$.  The proof is similar.
	
	$\bullet$ If we remove the $\sup_{\psi\in \Psi_Y}$, then
	our proof works without an ergodicity assumption on $f$.
	This simpler  result was also obtained recently  in \cite{FFK16}, and prior to this, related results were obtained in \cite{FH15a, GT12b}. But none of these results  allows to treat the more complicated setup with the supremum over the set  $\Psi_Y$ and this is crucial for our purposes.
\end{remark}

\subsection{Step 3 (Reduction to non-trivial nilcharacters)}
Since $\phi$ is an $s$-step nilsequence, there exist an $s$-step nilmanifold $X=G/\Gamma$, an ergodic nilrotation  $b\in G$, and a function $\Phi\in C(X)$, such that   $\phi(n)=\Phi(b^n\cdot e_X), n\in \mathbb{N}$. Since the linear span of vertical nilcharacters of $X$
are dense $C(X)$ (see Section~\ref{SS:nilcharacters}) we can assume that $\Phi$
is a vertical nilcharacter of $X$.

If   $\Phi$ is a trivial nilcharacter of $X$, then it factorizes through the nilmanifold
$$
X'=G/(G_s\Gamma)=(G/G_s)/((\Gamma\cap G_s)/G_s).
$$
The group $G/G_s$ is $(s-1)$-step nilpotent and $X'$ is an $(s-1)$-step nilmanifold.   Writing $b'$ for the image of $b$ in $G/G_s$, we have that $\phi(n)=\Phi'(b'^n\cdot e_{X'}), n\in\mathbb{N},$
for some $\Phi'\in C(X')$. We deduce that  $\phi$ is an $(s-1)$-step nilsequence.
Moreover, since the sequence $(f(n))_{n\in\mathbb{N}}$ is  ergodic for Ces\`aro averages on
$\mathbf{I}=([N_k])_{k\in\mathbb{N}}$, the induction hypothesis
gives that
$\norm{f}_{U^{s}(\mathbf{I})}=0$. Hence, the following  theorem  (which does not require ergodicity assumptions) implies that
the conclusion of Proposition~\ref{P:Discorrelation'} holds when the function $\Phi$ defining the nilsequence $\phi$ is
a trivial nilcharacter of $X$:
\begin{lemma}\label{L:Direct}
	Let $s\geq 2$ and $a\in \ell^\infty(\mathbb{N})$ be a sequence that admits correlations  for Ces\`aro averages  on the sequence of intervals
	$\mathbf{I}=(I_N)_{N\in\mathbb{N}}$. Suppose that $\norm{a}_{U^{s}(\mathbf{I})}=0$.
	Then
	for every   $(s-1)$-step nilsequence $\phi$ and every $(s-1)$-step nilmanifold $Y$,  we have
	$$
	\lim_{M\to \infty}\limsup_{N\to\infty} \mathbb{E}_{n\in I_N}\sup_{\psi\in \Psi_Y} |\mathbb{E}_{m\in [n,n+M]}\, a(m) \, \phi(m)\, \psi(m)|=0.
	$$
\end{lemma}
\begin{remark}
	A variant for logarithmic averages also holds  where one  replaces
	$\norm{a}_{U^{s}(\mathbf{I})}$ with    $\norm{a}_{U_{\text{log}}^{s}(\mathbf{I})}$
	and
	$\mathbb{E}_{n\in I_N}$ with $\mathbb{E}^{\log}_{n\in I_N}$. The proof is similar.
\end{remark}
\begin{proof}
	First notice that since every $(s-1)$-step nilsequence $\phi$  can be uniformly approximated by $(s-1)$-step nilsequences
	defined by functions with bounded Lipschitz norm,  the sequence $\phi$ can be absorbed in the $\sup$ (upon enlarging the nilmanifold $Y$). Hence, we can assume that $\phi=1$.

	By Lemma~\ref{L:nilkey} it suffices to show that
	$$
	\lim_{M\to \infty}\limsup_{N\to\infty} \mathbb{E}_{n\in I_N} \Big|\mathbb{E}_{m\in [M]} a(m+n)  \int  F_{0,M,n} \cdot T_{M,n}^{k_{1} m}F_{1,M,n}\cdot  \ldots \cdot  T_{M,n}^{k_{s-1} m}F_{s-1,M,n}\, d\mu_{M,n}\Big|=0,
	$$
	where the integers $k_1,\ldots, k_{s-1}$ are arbitrary,
	and for $M,n\in \mathbb{N}$ we have that $(X_{M,n},\mathcal{X}_{M,n}$, $\mu_{M,n},T_{M,n})$ is a system and the functions  $F_{0,M,n},\ldots, F_{s,M,n}\in L^\infty(\mu_{M,n})$
	are bounded by $1$. In fact, we will show by induction on $s\in \mathbb{N}$ a stronger statement, namely,  that  if the functions and
	the sequence are as above and they are all  bounded by $1$, then
	\begin{equation}\label{E:s-1}
		\limsup_{M\to \infty}\limsup_{N\to\infty} \mathbb{E}_{n\in I_N} \norm{\mathbb{E}_{m\in [M]}\, a(m+n) \, \,  T_{M,n}^{k_{1} m}F_{1,M,n}\cdots   T_{M,n}^{k_{s-1} m}F_{s-1,M,n}}_{L^2(\mu_{M,n})}  \leq 4\, \norm{a}_{U^{s}(\mathbf{I})}.
	\end{equation}
	We implicitly assume that the $s=1$ case corresponds to an estimate where there are no functions on the left hand side. So in order to verify the base case, we need to show that
	$$
	\limsup_{M\to \infty}\limsup_{N\to\infty} \mathbb{E}_{n\in I_N} \big|\mathbb{E}_{m\in [M]}\, a(m+n) \big|\leq  4\, \norm{a}_{U^{1}(\mathbf{I})}.
	$$
	To this end, we apply the van der Corput lemma for complex numbers. We get  for all $M,R\in \mathbb{N}$ with  $R\leq M$   and all  $n\in \mathbb{N}$,  that
	$$|\mathbb{E}_{m\in [M]} a(n+m)|^2
	\leq  4\, \mathbb{E}_{r\in [R]}  (1-rR^{-1})\big( \Re (\mathbb{E}_{m\in [M]} a(m+n+r) \cdot \overline{a(m+n)}) +R^{-1}+RM^{-1}\big).
	$$
	Hence,
	\begin{multline*}
		\limsup_{M\to \infty} \limsup_{N\to \infty} (\mathbb{E}_{n\in I_N} |\mathbb{E}_{m\in [M]} a(n+m)|)^2
		\leq \\4 \limsup_{R\to
			\infty}\limsup_{M\to
			\infty}\mathbb{E}_{r\in [R]} (1-rR^{-1}) \mathbb{E}_{m\in [M]}\Re (\mathbb{E}_{n\in \mathbf{I}} \, a(m+n+r)\cdot \overline{a(m+n)})= \\
		4 \limsup_{R\to
			\infty}\mathbb{E}_{r\in [R]} (1-rR^{-1}) \Re (\mathbb{E}_{n\in \mathbf{I}} \, a(n+r)\cdot \overline{a(n)})\leq\\
		4 \limsup_{R\to
			\infty}\mathbb{E}_{r\in [R]}  \Re (\mathbb{E}_{n\in \mathbf{I}} \, a(n+r)\cdot \overline{a(n)})
		=4\norm{a}_{U^1(\mathbf{I})}^2,
	\end{multline*}
	where the last estimate holds because
	$\mathbb{E}_{r\in [R]} (1-rR^{-1}) \Re(\mathbb{E}_{n\in \mathbf{I}} \, a(n+r)\cdot \overline{a(n)})$
	is the Ces\`aro average of
	$\mathbb{E}_{r\in [R]}\Re(\mathbb{E}_{n\in \mathbf{I}} \, a(n+r)\cdot \overline{a(n)})$ with respect to $R$.
	Hence, the asserted estimate holds.

	Suppose now  that the estimate  \eqref{E:s-1} holds for $s-1\in \mathbb{N}$, where $s\geq 2$,  we will show that it holds for $s$.
	We apply the  van der Corput lemma in the Hilbert space $L^2(\mu_{M,n})$ and then use the
	Cauchy-Schwarz inequality.  We get  for all $M, R \in \mathbb{N}$ with  $R\leq M$   and all  $n\in \mathbb{N}$,  that
	\begin{multline*}
		\norm{\mathbb{E}_{m\in [M]}\, a(m+n) \, \,  T_{M,n}^{k_{1} m}F_{1,M,n}\cdots   T_{M,n}^{k_{s-1} m}F_{s-1,M,n}}_{L^2(\mu_{M,n})}^2\leq\\
		4\, \mathbb{E}_{r\in [R]}  \norm{\mathbb{E}_{m\in [M]} a(m+n+r) \cdot \overline{a(m+n)}\,\,   T_{M,n}^{\tilde{k}_{1} m}\tilde{F}_{1,M,n,r}\cdots   T_{M,n}^{\tilde{k}_{s-1} m}\tilde{F}_{s-2, M,n,r}}_{L^2(\mu_{M,n})}
		+R^{-1}+RM^{-1},
	\end{multline*}
	where $\tilde{F}_{j,M,n,r}:=T_{M,n}^{k_jr}F_{j,M,n}\cdot \overline{F_{j,M,n}}$,  $\tilde{k}_j:=k_j-k_{s-1}$,  for $j=1,\ldots, s-2$.

	Hence,
	the square of the left hand side in \eqref{E:s-1} is bounded by
	\begin{multline*}
		4 \, \limsup_{R\to \infty}\mathbb{E}_{r\in [R]} \limsup_{M\to \infty}\limsup_{N\to\infty} \mathbb{E}_{n\in I_N}  \\ \norm{\mathbb{E}_{m\in [M]} a(m+n+r) \cdot \overline{a(m+n)}\, \,  T_{M,n}^{\tilde{k}_{1} m}\tilde{F}_{1,M,n,r}\cdots   T_{M,n}^{\tilde{k}_{s-1} m}\tilde{F}_{s-2,M,n,r}}_{L^2(\mu_{M,n})},
	\end{multline*}
	where $\tilde{F}_{j,M,n,r}\in L^\infty(\mu_{M,n})$ are  bounded by $1$ and $\tilde{k}_j$ are integers for $j=1,\ldots,s-2$.
	
	Applying the induction hypothesis for the sequences $S_ra\cdot \overline{a}$, $r\in\mathbb{N}$ (which also admit correlations
	for Ces\`aro averages  on $\mathbf{I}$ and are bounded by $1$),  the functions $\tilde{F}_{j,M,n,r}$,
	and the integers $\tilde{k}_{j}$, $j=1,\ldots, s-2$,
	and averaging over $r\in \mathbb{N}$, we deduce that the last expression is bounded by $16$ times
	$$
	\limsup_{R\to \infty}\mathbb{E}_{r\in [R]}\norm{S_ra\cdot \overline{a}}_{U^{s-1}(\mathbf{I})}\leq
	\big(\limsup_{R\to \infty}\mathbb{E}_{r\in [R]}\norm{S_ra\cdot \overline{a}}_{U^{s-1}(\mathbf{I})}^{2^{s-1}}\big)^{\frac{1}{2^{s-1}}}= \norm{a}_{U^{s}(\mathbf{I})}^2.
	$$
	Taking square roots, we deduce that \eqref{E:s-1} holds, completing the induction.
\end{proof}

Hence, it  suffices to consider  the case where $\Phi$ is a non-trivial vertical nilcharacter of $X$,
and we have thus reduced matters to proving the following result (note that strong aperiodicity and ergodicity are no longer needed):
\begin{proposition}\label{P:Discorrelation''}
	Let $s\geq 2$ and $f\in \mathcal{M}$ be a multiplicative function. Let
	$X=G/\Gamma$ be an $s$-step nilmanifold, $b\in G$ be
	an ergodic nilrotation, and
	$\Phi\colon X\to \mathbb{C}$ be a non-trivial vertical nilcharacter.
	Then
	for every   $(s-1)$-step nilmanifold $Y$ we have
	$$
	\lim_{M\to \infty}\limsup_{N\to\infty} \mathbb{E}_{n\in [N]}\sup_{\psi\in \Psi_Y} |\mathbb{E}_{m\in [n,n+M]}\, f(m)\,\Phi(b^m\cdot e_X)\, \psi(m)|=0.
	$$
\end{proposition}
\begin{remarks}
	$\bullet$ Note that in this statement we do not impose any assumption on $f\in \mathcal{M}$.
	
	$\bullet$ If $(a(M,n))_{M,n\in\mathbb{N}}$ is  bounded and
	$\lim_{M\to \infty}\limsup_{N\to\infty} \mathbb{E}_{n\in [N]}\, |a(M,n)|=0$, we then have that
	$\lim_{M\to \infty}\limsup_{N\to\infty} \mathbb{E}^{\log}_{n\in [N]}\, |a(M,n)|=0$.
	It follows that  Proposition~\ref{P:Discorrelation''}    also holds  with  the averages
	$\mathbb{E}^{\log}_{n\in [N]}$ in place of the averages  $\mathbb{E}_{n\in [N]}$.
\end{remarks}

\subsection{Step 4 (Disjointifying the intervals $[n,n+M]$)}
Suppose that  Proposition~\ref{P:Discorrelation''} fails. Then there exist  a   multiplicative function $f\in \mathcal{M}$ and
\begin{itemize}
	\item $\varepsilon>0$;

	\item
	strictly increasing  sequences of integers
	$(M_k)$ and $(N_k)$ satisfying $M_k/N_k\to 0$ as $k\to \infty$;
	
	\item an $s$-step  nilmanifold $X=G/\Gamma$ and an ergodic nilrotation
	$b\in G$;
	
	\item a non-trivial vertical nilcharacter $\Phi$ of $X$;

	\item an $(s-1)$-step nilmanifold $Y$ and  $(s-1)$-step nilsequences $\psi_{k,n}\in \Psi_Y$,   $k,n\in \mathbb{N}$;
\end{itemize}
such that
$$
\mathbb{E}_{n\in [N_k]}|\mathbb{E}_{m\in [n,n+M_k)}\, f(m)\, \phi(m)\, \psi_{k,n}(m)|> \varepsilon
$$
for all large enough  $k\in \mathbb{N}$, where
$$
\phi(m):=\Phi(b^m\cdot e_X), \quad m\in \mathbb{N}.
$$

We follow the argument used in the proof of  \cite[Theorem~4]{ALR15} in order to disjointify the intervals $[n,n+M]$.
Since   $M_k/N_k\to 0$, we have for every bounded sequence $(a(k,n))_{k,n\in\mathbb{N}}$ that
$$
\lim_{k\to\infty} \big(\mathbb{E}_{n\in [N_k]}\, a(k,n)-\mathbb{E}_{r\in [M_k]}\mathbb{E}_{n\in [N_k]\colon n\equiv r \! \! \! \! \! \pmod{M_k}} \, a(k,n)\big)=0.
$$
Applying this for the sequence
$$
a(k,n):=|\mathbb{E}_{m\in [n,n+M_k)}\, f(m)\, \phi(m)\, \psi_{k,n}(m)|, \quad k,n\in\mathbb{N},
$$
we deduce that for every large enough $k\in \mathbb{N}$ there exists $r_k\in [M_k]$ such that
$$
\mathbb{E}_{ n\in [N_k]\colon n\equiv r_k \! \! \! \! \pmod{M_k}}|\mathbb{E}_{m\in [n,n+M_k)}\, f(m)\, \phi(m)\, \psi_{k,n}(m)|> \varepsilon.
$$
Upon changing  the sequences $(\psi_{k,n}(m))_{m\in\mathbb{N}}$ by  multiplicative constants of modulus $1$ that depend on $k$ and $n$ only, we can remove the norm in the previous estimate. Hence, without loss of generality, we can assume
for all large enough  $k\in \mathbb{N}$ that
$$
\mathbb{E}_{ n\in [N_k]\colon n\equiv r_k \! \! \! \! \pmod{M_k}}\big(\mathbb{E}_{m\in [n,n+M_k)}\, f(m)\, \phi(m)\, \psi_{k,n}(m)\big)> \varepsilon.
$$
Since  $M_k/N_k\to 0$, we deduce from the last estimate that
$$
\mathbb{E}_{n \in [N_k]}\, f(n)\,  g_k(n)> \varepsilon
$$
for all large enough  $k\in \mathbb{N}$,
where
\begin{equation}\label{E:gk}
	g_k(n):=\sum_{\ell=1}^\infty \,  {\bf 1}_{[(\ell-1)M_k, \ell  M_k)}(n) \, \phi(n)\, \psi_{k,\ell}(n), \quad k,n\in \mathbb{N},
\end{equation}
for some   $(s-1)$-step nilsequences $\psi_{k,\ell}\in \Psi_Y$, $k, \ell\in \mathbb{N}$.

Hence, in order to get a contradiction and complete the proof of Proposition~\ref{P:Discorrelation''}, it remains to verify that
\begin{equation}\label{E:fgk}
\lim_{k\to\infty} \mathbb{E}_{n \in [N_k]}\, f(n)\,  g_k(n) =0.
\end{equation}
The only property to be used for the intervals $[(\ell-1)M_k, \ell  M_k)$ is that their lengths tend to infinity as $k\to \infty$ uniformly in $\ell\in\mathbb{N}$.

\subsection{Step 5 (Applying an orthogonality criterion)}
We will use the following orthogonality criterion
for multiplicative functions in $\mathcal{M}$:
\begin{lemma}[K\'atai~\cite{K86}, see also  \cite{D74, MV77}]
	For every $\varepsilon>0$
	there exist
	$\delta=\delta(\varepsilon)>0$ and $K:=K(\varepsilon)$ such that the
	following holds: If $N\geq K$  and $a\colon [N]\to \mathbb{C}$ is a finite sequence that is bounded by $1$  and satisfies
	$$
	\max_{\substack{p,q\text{ \rm primes}\\1<p<q<K}}\bigl|\mathbb{E}_{n\in[\lfloor
		N/q\rfloor]} a(pn)\, \overline{a(qn)}\bigr|< \delta,
	$$
	then
	$$
	\sup_{ f \in\mathcal{M}}\bigl|\mathbb{E}_{n\in[N]} f (n)\, a(n)\bigr|<\varepsilon.
	$$
\end{lemma}
\begin{remark}
	Using  Vinogradov's  bilinear  method  one can obtain quantitatively superior refinements, see for example \cite[Theorem~2]{BSZ12}.
\end{remark}
We deduce from this lemma  the following:
\begin{lemma} \label{L:katai}
	Let $N_k\to \infty$  be integers and   $a_k\colon [N_k]\to \mathbb{C}$, $k\in\mathbb{N}$,
	be finite sequences that are  bounded by $1$ and satisfy 	
	
	$$
	\lim_{k\to\infty}\mathbb{E}_{n\in [cN_k]} \, a_k(pn)\, \overline{a_k(p'n)}=0
	$$
	for every $p,p'\in \mathbb{N}$ with $p\neq p'$ and every $c>0$.
	Then
	$$
	\lim_{k\to \infty}\sup_{ f \in\mathcal{M}}\big|\mathbb{E}_{n\in [N_k]} \, f(n)\, a_k(n)\big|=0.
	$$
\end{lemma}
Applying Lemma~\ref{L:katai},  we get that in order to prove \eqref{E:fgk} and complete the proof of Proposition~\ref{P:Discorrelation''},  it suffices to show that for  every $p,p'\in \mathbb{N}$ with
$p\neq p'$  and every $c>0$,  we have
$$
\lim_{k\to\infty} \mathbb{E}_{n \in [cN_k]}\, g_k(pn)\,  \overline{g_k(p'n)} =0
$$
where $g_k$ is
as in \eqref{E:gk}.
Equivalently, we have to show that (the sum below is finite)
$$
\lim_{k\to\infty} \mathbb{E}_{n \in [cN_k]}\,\Big(\sum_{\ell, \ell'\in \mathbb{N}} \, {\bf 1}_{I_{k, \ell,\ell'}}(n)\,  \phi(pn)\, \psi_{k,\ell}(pn) \, \overline{\phi(p'n)}\, \overline{\psi_{k,\ell'}(p'n)} \Big)=0,
$$
where
$$
I_{k, \ell,\ell'}:=\Big[\frac{(\ell-1)M_k}{p}, \frac{\ell  M_k}{p}\Big)\, \bigcap \, \Big[\frac{(\ell'-1)M_k}{p'}, \frac{\ell'  M_k}{p'}\Big), \quad k, \ell, \ell'\in \mathbb{N}.
$$
Note that for fixed $k\in \mathbb{N}$ the intervals $I_{k, \ell,\ell'}$, $\ell,\ell'\in \mathbb{N},$ are disjoint (and some of them  empty). Since $M_k\to\infty$, they partition the interval
$[cN_k]$ into subintervals $J_{k,l}$,  $l=1,\ldots, L_k$,  with $L_k\to\infty$ and $\min_{l\in [L_k]}|J_{k,l}|\to \infty$ as $k\to \infty$, and a set $Z_k$ with $|Z_k|/N_k\to 0$ as $k\to \infty$.
Since $|Z_k|/N_k\to 0$ as $k\to \infty$, it suffices to show that
$$
\lim_{k\to\infty}  \mathbb{E}_{n \in [cN_k]}\,\Big( \sum_{l\in[L_k]} \, {\bf 1}_{J_{k, l}}(n)\,  \phi(pn) \, \overline{\phi(p'n)}\, \psi_{k,l}(pn) \, \overline{\psi'_{k,l}(p'n)}\Big) =0,
$$
where $\psi_{k,l}=\psi_{k,\ell}$, $\psi'_{k,l}=\psi_{k,\ell'}$
for some $\ell=\ell(l)$, $\ell'=\ell'(l)$. Since
$\min_{l\in [L_k]}|J_{k,l}|\to \infty$ as $k\to \infty$, it suffices to show that if  $J'_k$, $k\in\mathbb{N}$,  are intervals with
$|J'_k|\to\infty$ as $k\to \infty$, then for every $(s-1)$-step nilmanifold $Y$ we have
$$
\lim_{k\to\infty}  \sup_{\psi\in \Psi_Y}|\mathbb{E}_{n \in J'_k}\,  \phi(pn) \, \overline{\phi(p'n)}\, \psi(n)| =0.
$$

Thus, in order to complete the proof of  Proposition~\ref{P:Discorrelation''}
it suffices to verify the following asymptotic orthogonality property which has purely dynamical context:
\begin{proposition}\label{P:orthogonality}
	For $s\geq 2$ let $X=G/\Gamma$ be an $s$-step nilmanifold and  $b\in G$ be an ergodic nilrotation. Furthermore, let  $\Phi, \Phi'$ be non-trivial vertical  nilcharacters of $X$  with the same frequency.
	Then for every $p,p'\in \mathbb{N}$ with $p\neq p'$, every sequence of intervals  $(I_N)_{N\in\mathbb{N}}$ with $|I_N|\to \infty$, and every $(s-1)$-step nilmanifold $Y$, we have
	\begin{equation}\label{E:keypp'}
		\lim_{N\to \infty} \sup_{ \psi\in \Psi_Y}|\mathbb{E}_{n\in I_N}\,  \Phi(b^{pn}\cdot e_X) \,
		\overline{\Phi'(b^{p'n}\cdot e_X)}\, \psi(n)|=0.
	\end{equation}
\end{proposition}
A model case is when  $ \Phi(b^n\cdot e_X) = \Phi'(b^n\cdot e_X)= \e(n^s\beta)$ with $\beta$ irrational.  Then the  statement to be proved reduces to
$$
\lim_{N\to \infty} \sup_{ \psi\in \Psi_Y}|\mathbb{E}_{n\in I_N}\, \e(n^s\alpha)\, \psi(n)|=0
$$
where $\alpha:=(p^s-q^s)\beta$ is irrational.
This can be verified easily by using van der Corput's lemma for complex numbers and Lemma~\ref{L:nilkey}.  The proof in the general case  is  much harder
though; it  is given in the next subsection.

\subsection{Step 6 (Proof of the dynamical property)}
The goal of this last subsection is to prove Proposition~\ref{P:orthogonality}. Let us remark first that  although we were not able to adapt a  related  argument in \cite[Theorem~6.1]{FH15a} to the current setup, we found some of the ideas used there very useful.

The main idea is as follows. We apply the van der Corput lemma for complex numbers $(s-1)$ times in order to
cancel out the term $\psi$ (Lemma~\ref{L:nilkey} is useful in this regard) and
we reduce \eqref{E:keypp'}
to verifying  $U^s(\mathbf{I})$-uniformity for the sequence
$\big(\Phi(b^{pn}\cdot e_X) \,
\overline{\Phi'(b^{p'n}\cdot e_X)}\big)_{n\in\mathbb{N}}$.  The fact that the supremum over $\Psi_Y$ no longer appears
has the additional advantage that we  only need to use qualitative (and not quantitative) equidistribution results on nilmanifolds.

The key in obtaining the necessary  $U^s(\mathbf{I})$-uniformity
is to establish  that the nilcharacter  $\Phi\otimes \Phi'$ is non-trivial  on the $s$-step nilmanifold $W:=\overline{\{(b^{pn}\cdot e_X, b^{p'n}\cdot e_X), n\in\mathbb{N}\}}$. Although the precise  structure of
the  nilmanifold $W$ is  very difficult to determine (and depends on the choice of the ergodic nilrotation $b$) it is possible to extract  partial information on $W$ that suffices for our purposes.  This last idea is taken from the proof of \cite[Proposition~6.1]{FH15a} and the precise statement is as follows:

\begin{proposition}\label{P:keyinv}
	For $s\in \mathbb{N}$ let $X=G/\Gamma$ be a connected $s$-step nilmanifold and $b\in G$ be an ergodic nilrotation. Let  $p,p'\in\mathbb{N}$ be distinct and let $W$ be the closure of the sequence  $(b^{pn}\cdot e_X, b^{p'n}\cdot e_X)_{n\in \mathbb{N}}$ in $X\times X$. Then $W$ is a nilmanifold that can be represented as $W=H/\Delta$ where
	$\Delta=\Gamma\times\Gamma$ and  $H$ is a subgroup of $G\times G$  such that
	$(b^p,b^{p'})\in H$ and
	\begin{equation}\label{E:invariance}
		(u^{p^s}, u^{p'^s})\in H_s \qquad \text{for every }  u\in G_s.
	\end{equation}
\end{proposition}
\begin{proof}
	As remarked in Section~\ref{SS:equidistribution},  $W$ is a sub-nilmanifold of $X\times X$.
	Let
	$$
	H:=\overline{\langle  \Gamma\times\Gamma, (b^{p},b^{p'})\rangle},
	$$
	that is, $H$ is the smallest closed subgroup of $G\times G$ that contains
	$\Gamma\times\Gamma$ and  $(b^{p},b^{p'})$.
	\begin{claim*}
		We have that	$W=H/(\Gamma\times\Gamma).$
	\end{claim*}
	Indeed, by the definition of $H$ and $W$, we have   $H\cdot(e_X,e_X)\supset W$.
	Furthermore, as remarked in Section~\ref{SS:equidistribution},  we have
	$W=H_1\cdot(e_X,e_X)$ for some closed subgroup $H_1$ of $G\times G$ containing the  element   $(b^{p},b^{p'})$.  Since $W$ is compact,
	the set  $H_2:=H_1\cdot(\Gamma\times\Gamma)$ is  closed in $G$ (see remarks in Section~\ref{SS:nilmanifolds}). Since $H_2$ is a closed subgroup that contains  $(b^{p},b^{p'})$  and $\Gamma\times\Gamma$, we have  $H_2\supset H$, hence
	$H\cdot(e_X,e_X)\subset H_2\cdot(e_X,e_X)=H_1\cdot(e_X,e_X)=W$. Therefore, $H\cdot(e_X,e_X)= W$,
	which implies that  $W=H/(\Gamma\times\Gamma)$. This proves the claim.

	It remains to show \eqref{E:invariance}.
	To this end, it suffices to verify the following:
	\begin{claim*}
		\label{cl:Lj}
		Let  $j\in \{1,\ldots, s\}$ and  $g\in G_j$. Then
		$(g^{p^j}, g^{p'^j})\in  H_j\cdot (G_{j+1}\times G_{j+1})$.
	\end{claim*}
	We proceed by induction on $j$. We first prove the claim for $j=1$.
	Let $Z:=G/(G_2\Gamma)$ and $\pi\colon G\to Z$ be the natural projection. Since $X$ is connected, the nilmanifold $Z$ is connected. It is also a compact Abelian group, hence $Z$ is a torus.  The ergodic nilrotation  $b$  projects to the element $\beta:=\pi(b)$ and every power of $\beta$  acts ergodically on the torus $Z$. Since $(\pi\times \pi)(H)$  is the closure of the sequence $(\beta^{pn},\beta^{p'n})$ in $Z\times Z$, it follows that  $(\pi\times \pi)(H)=\{(z^p,z^{p'}) , z\in Z\}$.
	We deduce from this (recall also that  $\Gamma\times\Gamma\subset H$) that
	$$
	\{(g^p,g^{p'}),\,  g\in G\}\subset H\cdot (G_2\times G_2),
	$$
	which proves the claim for $j=1$.

	Let now $j\geq 2$ and suppose that  the claim holds for $j-1$, we will show that it holds for $j$. To this end, notice first that if  $g,g'\in G_j$ and $u_1,u'_1,u_2,u'_2\in G_{j+1}$, then
	$$
	g^{p^j}u_1\cdot g'^{p^j}u'_1=(gg')^{p^j}\bmod G_{j+1}, \quad
	g^{p'^j}u_2\cdot g'^{p'^j}u'_2=(gg')^{p'^j}\bmod G_{j+1}.
	$$
	It follows that the family of elements $g$ of $G_j$ for which the conclusion of the claim  holds is a subgroup of $G_j$. Therefore, it suffices to prove that the conclusion of the claim holds when
	$g$ is a commutator, that is, $g=[h,v]$ for some $h\in G$ and some $v\in G_{j-1}$.
	By the $j=1$ case of the claim   there exist $u,u'\in G_2$ such that $(h^{p}u,h^{p'}u')\in H$ and by the induction hypothesis there exist $w,w'\in G_j$ such that
	$(v^{p^{j-1}}w,v^{p'^{j-1}}w')\in H_{j-1}$. Then the  commutator
	$\bigl([h^{p}u,v^{p^{j-1}}w]\,,\, [h^{p'}u',v^{p'^{j-1}}w']\bigr)$
	of these elements belongs to $H_j$. Since $$
	[h^{p}u,v^{p^{j-1}}w]=[h,v]^{p^j}=
	g^{p^j}\bmod G_{j+1}, \quad
	[h^{p'}u',v^{p'^{j-1}}w']=[h,v]^{p'^j}=
	g^{p'^j}\bmod G_{j+1},
	$$
	we deduce that
	$$
	(g^{p^j}, g^{p'^j})\in  H_j\cdot (G_{j+1}\times G_{j+1}).
	$$
	This completes the proof of   the claim.
	
	Setting $j=s$ and using that $G_{s+1}$ is trivial we deduce \eqref{E:invariance}, completing the proof.
\end{proof}

\begin{lemma}\label{L:semi}
	For $s\geq 2$ let $W=H/\Delta$ be an $s$-step nilmanifold and $h\in H$ be an ergodic nilrotation. Let
	$\Phi$ be a non-trivial vertical nilcharacter of $W$  and
	$$
	\phi(n):=\Phi(h^n\cdot e_W), \quad n\in\mathbb{N}.
	$$
	Then $\norm{\phi}_{U^s(\mathbf{I})}=0$ for every   sequence of intervals $\mathbf{I}=(I_N)_{N\in\mathbb{N}}$
	with $|I_N|\to \infty$.
\end{lemma}
\begin{proof}
	As remarked in Section~\ref{SS:equidistribution}, we have  $\norm{\phi}_{U^s(\mathbf{I})}=\nnorm{\Phi}_{s}$, where the seminorm is computed
	for the system induced on $W$ with the normalized  Haar measure $m_W$ by the ergodic nilrotation by $h\in H$.
	Let $\mathcal{Z}_{s-1}(W)$ be defined as in Section~\ref{SS:GHKseminorms}.
	
	It is implicit in \cite[Theorem~13.1]{HK05a} and also follows by combining   \cite[Lemma~4.5]{Zi07} and \cite{Lei05}, that $L^2(\mathcal{Z}_{s-1}(W))$  consists
	exactly of those functions in $L^2(m_W)$ that are $H_s$-invariant (\cite{Lei05} shows that  the factors $\mathcal{Z}_s$ and  $Y_s$  defined in \cite{HK05a} and  \cite{Zi07} respectively are the same).
	Since   $\Phi$ is a non-trivial  vertical nilcharacter of $W$, it is orthogonal to any $H_s$-invariant function in $L^2(m_W)$, hence  $\Phi$ is orthogonal to any function  in  $L^2(\mathcal{Z}_{s-1}(W))$. As remarked in Section~\ref{SS:GHKseminorms}, this implies that $\nnorm{\Phi}_s=0$ and completes the proof.
\end{proof}

\begin{proposition}\label{C:unfm}
	For $s\geq 2$ let  $W=H/\Delta$ be an $s$-step nilmanifold and $h\in H$
	be an ergodic nilrotation.  Furthermore, let $\Phi$ be a non-trivial vertical  nilcharacter of $W$.
	Then for every   sequence of intervals $\mathbf{I}=(I_N)_{N\in\mathbb{N}}$
	with $|I_N|\to \infty$   and every
	$(s-1)$-step nilmanifold $Y$ we have
	$$
	\lim_{N\to\infty}\sup_{\psi\in \Psi_Y}
	|\mathbb{E}_{n\in I_N}\, \Phi(h^n\cdot e_W)\, \psi(n)|=0.
	$$	
\end{proposition}
\begin{proof}
	Let $\phi(n):=\Phi(h^n\cdot e_W), n\in\mathbb{N}$.
	By Lemma~\ref{L:semi} we  have that $ \norm{\phi}_{U^s(\mathbf{I})}=0$.
	
	It follows from  Lemma~\ref{L:nilkey} that it suffices to establish the following:
	Let  $s\in \mathbb{N}$,   $\mathbf{I}=(I_N)_{N\in\mathbb{N}}$ be a sequence of intervals with $|I_N|\to \infty$,  and $a\in \ell^\infty(\mathbb{N})$ be a sequence that admits correlations  for Ces\`aro averages  on $\mathbf{I}$. Furthermore, for $N\in\mathbb{N}$, let $(X_N,\mathcal{X}_N,\mu_N,T_N)$ be a system,  $F_{0,N},\ldots, F_{s-1,N}\in L^\infty(\mu_N)$ be functions
	bounded by $1$, and let $k_1,\ldots, k_{s-1}\in \mathbb{Z}$.
	Then we have
	$$
	\limsup_{N\to\infty}
	\Big|\mathbb{E}_{n\in I_N}\, a(n)\,\int  F_{0,N} \cdot T_N^{k_{1} n}F_{1,N} \cdot  \ldots \cdot  T_N^{k_{s-1} n}F_{s-1,N}\ d\mu_N \Big|\leq 4 \norm{a}_{U^s(\mathbf{I})}.
	$$
	
	This estimate can be proved by induction on $s$ in a rather standard way using the van der Corput lemma for inner product spaces, the details are given in \cite[Section~2.3.1]{Fr15}.
\end{proof}
We are now ready to prove Proposition~\ref{P:orthogonality} which is the last step in the proof of Theorem~\ref{T:uniformity}.
\begin{proof}[Proof of Proposition~\ref{P:orthogonality}] We argue by contradiction. Suppose that for some  $s\geq 2$ there exist an
	$s$-step nilmanifold  $X=G/\Gamma$, an ergodic nilrotation   $b\in G$, non-trivial vertical nilcharacters $\Phi, \Phi'$ of $X$  with the  same frequency,
	$p,p'\in \mathbb{N}$ with $p\neq p'$, a sequence of intervals  $(I_N)_{N\in\mathbb{N}}$ with $|I_N|\to \infty$, and an $(s-1)$-step nilmanifold $Y$, such that
	\begin{equation}\label{E:tobecontr}
		\limsup_{N\to \infty} \sup_{ \psi\in \Psi_Y}|\mathbb{E}_{n\in I_N}\,  \Phi(b^{pn}\cdot e_X) \,
		\overline{\Phi'(b^{p'n}\cdot e_X)}\, \psi(n)|>0.
	\end{equation}
	
	We first reduce matters to the case where the nilmanifold $X$ is connected.  As remarked in Section~\ref{SS:equidistribution}, there exists $r\in \mathbb{N}$ such that  $b^r$ acts ergodically on the connected component $X_0$ of the nilmanifold $X$. Then for some $j\in \{0,\ldots, r-1\}$ we have
	$$
	\limsup_{N\to \infty} \sup_{ \psi\in \Psi_Y}|\mathbb{E}_{n\in I_N}\,  \Phi(b^{p(rn+j)}\cdot e_X) \,
	\overline{\Phi'(b^{p'(rn+j)}\cdot e_X)}\, \psi(rn+j)|>0.
	$$
	Since $(\psi(rn+j))_{n\in\mathbb{N}}\in \Psi_Y$ for every $r,s\in \mathbb{N}$, we deduce that
	$$
	\limsup_{N\to \infty} \sup_{ \psi\in \Psi_Y}|\mathbb{E}_{n\in I_N}\,  \tilde{\Phi}(\tilde{b}^{pn}\cdot e_X) \,
	\overline{\tilde{\Phi}'(\tilde{b}^{p'n}\cdot e_X)}\, \psi(n)|>0,
	$$
	where $\tilde{b}:=b^r$ is an ergodic nilrotation of  $X_0$, and the functions $\tilde{\Phi}, \tilde{\Phi}'\colon X_0\to \mathbb{C}$,  defined by  $\tilde{\Phi}(x):=\Phi(b^{pj}x)$,
	$\tilde{\Phi}'(x):=\Phi(b^{p'j}x)$, $x\in X_0$, are non-trivial vertical  nilcharacters of $X_0$ (see the discussion in Section~\ref{SS:nilcharacters})  with the same frequency.
	Hence, we can assume that the nilmanifold $X$ is connected.
	
	Let $h:=(b^p,b'^{p'})$.
	By the discussion in Section~\ref{SS:equidistribution} and Proposition~\ref{P:keyinv}, the element $h$ acts ergodically on a nilmanifold $W$ that can be represented as $W=H/\Delta$ where $H$ is a subgroup of $G\times G$  such that  $h\in H$ and
	$(u^{p^s}, u^{p'^s})\in H_s$ for every   $u\in G_s$.

	We will show that   the restriction of the function
	$\Phi\otimes \overline{\Phi'}$ on $W$  is a non-trivial vertical nilcharacter of $W$.
	To this end, we use our hypothesis that
	$$
	\Phi(u\cdot x)=\chi(u)\, \Phi(x)\ \text{ and }\ \Phi'(u\cdot x)=\chi(u)\, \Phi'(x)
	\ \text{ for } \ u\in G_s\ \text{ and }\ x\in X,
	$$
	where  $\chi$ is a
	non-trivial element of the dual of $G_s$ that is   $(G_s\cap \Gamma)$-invariant. Hence,
	$$
	(\Phi\otimes \overline{\Phi'})\big((u,u')\cdot (x,x')\big)=\chi(u)\,  \overline{\chi}(u')\, (\Phi\otimes \overline{\Phi'})(x,x'), \quad \text{ for }\  u,u'\in G_s \ \text{ and }\  x,x'\in X.
	$$
	Since  $H_s\subset G_s\times G_s$, it follows from this identity
	that $\Phi\otimes \overline{\Phi'}$ is a vertical nilcharacter of
	$W=H/\Delta$. It remains to show that $\chi\cdot \overline{\chi}$ is non-trivial on $H_s$.  Arguing by contradiction, suppose it is.  Since
	$(u^{p^s}, u^{p'^s})\in H_s$ for every   $u\in G_s$,
	we  get
	$$
	\chi(u^{p^s-p'^s})=\chi(u^{p^s}) \, \overline{\chi(u^{p'^s})}=1 \quad \text{for every } u\in G_s.
	$$
	Since $G_s$ is connected for $s\geq2$ and $p\neq p'$, the map $u\mapsto u^{p^s-p'^s}$ is onto $G_s$, hence $\chi$ is the trivial character on $G_s$, a contradiction.

	Combining the above, we get that  Proposition~\ref{C:unfm}
	applies and gives 	$$
	\lim_{N\to\infty}\sup_{\psi\in \Psi_Y}
	|\mathbb{E}_{n\in I_N}\,(\Phi\otimes \overline{\Phi'})(h^n\cdot e_W)\, \psi(n)|=0.
	$$	
	This contradicts \eqref{E:tobecontr} and completes the proof
	of Proposition~\ref{P:orthogonality}.
\end{proof}




\bibliographystyle{amsplain}

\begin{thebibliography}{99}
	
	\bibitem{AKLR14} H.~el Abdalaoui, J.~Ku{\l}aga-Przymus, M. Lema\'nczyk, T.~de la Rue.
	\newblock The Chowla and the Sarnak conjectures from ergodic theory point of view.
	\newblock {\em Discrete Contin. Dyn. Syst.} {\bf 37} (2017), no. 6, 2899--2944.
		
	\bibitem{ALR15} H.~el Abdalaoui, M.~Lema\'nczyk,  T.~Rue.
	\newblock Automorphisms with quasi-discrete spectrum, multiplicative functions and average orthogonality along short intervals.
    \newblock 	{\em Int. Math. Res. Not.}, Volume 2017, Issue 14, (2017),   4350--4368
	
	\bibitem{BHK05} V.~Bergelson, B.~Host, B.~Kra, with an appendix by I. Ruzsa.
	\newblock Multiple recurrence and nilsequences.
	\newblock {\em Inventiones Math.} {\bf 160} (2005), no. 2, 261--303.
	
	\bibitem{BSZ12} J.~Bourgain, P.~Sarnak, T.~Ziegler.
	\newblock Disjointness of M\"obius from horocycle flows.
	\newblock From Fourier analysis and number theory to Radon transforms and geometry. {\em Dev. Math.} {\bf 28}, Springer, New York,	(2013),	67--83.
	
	\bibitem{Ch65} S.~Chowla.
	\newblock The Riemann Hypothesis and Hilberts Tenth Problem.
	\newblock  {\em Mathematics and Its	Applications} {\bf 4},  Gordon and Breach Science Publishers, New York, 1965.
	
	\bibitem{CF11} Q.~Chu, N.~Franzikinakis.
	\newblock Pointwise convergence for cubic and polynomial ergodic averages of non-commuting transformations.
	\newblock {\em Ergodic Theory Dynam. Systems}  {\bf 32}  (2012), 877--897.
	
	\bibitem{CFH11} Q.~Chu, N.~Franzikinakis, B.~Host.
	\newblock Ergodic averages of commuting transformations with distinct degree polynomial iterates.
	\newblock {\em Proc. Lond. Math. Soc.} {\bf 102} (2011), 801--842.
	
	\bibitem{D74} H.~Daboussi.
	\newblock Fonctions multiplicatives presque p\'eriodiques B.
	\newblock D'apr\`es un travail commun avec H.~Delange. Journ\'ees Arithm\'etiques de Bordeaux (Conf., Univ. Bordeaux, Bordeaux, 1974), pp. 321--324. {\em Asterisque} {\bf  24-25} (1975), 321--324.
	
	\bibitem{El90} P.~Elliott.
	\newblock Multiplicative functions $|g|\leq 1$ and their convolutions: An overview.
	\newblock S\'eminaire de Th\'eorie des Nombres, Paris 1987-88.
	{\em Progress in Mathematics} {\bf  81} (1990), 63--75.
	
	\bibitem{El94} P.~Elliott.
	\newblock On the correlation of multiplicative and the sum of additive arithmetic functions.
	\newblock  {\em Mem. Amer. Math. Soc. } {\bf 112} (1994), no. 538, viii+88pp.
	
	\bibitem{FFK16} 	L.~Flaminio, K.~Fraczek, J.~Ku{\l}aga-Przymus, M. Lema\'nczyk.	
	\newblock Approximate orthogonality of powers for ergodic affine unipotent diffeomorphisms on nilmanifolds.
	\newblock (2016), \texttt{arXiv:1609.00699}.
	
	\bibitem{Fr15}
	N.~Frantzikinakis.
	\newblock Multiple correlation sequences and nilsequences.
	\newblock {\em Invent. Math.} {\bf 202} (2015), no. 2, 875--892.
	
	\bibitem{F16} N.~Frantzikinakis.
	\newblock An averaged Chowla and Elliott  conjecture along  independent polynomials.
	\newblock To appear in {\em Int. Math. Res. Not.}  \texttt{arXiv:1606.08420}.
	
	\bibitem{FH15a}
	N.~Frantzikinakis,   B.~Host.
	\newblock Higher order Fourier analysis of multiplicative functions and applications.
	\newblock {\em J. Amer. Math. Soc.} {\bf 30}  (2017), 67--157.
	
	\bibitem{FH17}
	N.~Frantzikinakis,   B.~Host.
	\newblock The logarithmic Sarnak conjecture for ergodic weights.
	\newblock (2017), 	\texttt{arXiv:1708.00677}.
	
	\bibitem{FHK2} N.~Frantzikinakis, B.~Host, B.~Kra.
	\newblock The polynomial multidimensional Szemer\'edi theorem along shifted primes.
	\newblock {\em Isr. J. Math.} {\bf 194} (2013), 331--348.
	
	\bibitem{Fu77} H.~Furstenberg.
    \newblock 	Ergodic behavior of diagonal measures and a theorem of Szemer\'edi on arithmetic progressions.
	\newblock {\em J. Analyse Math.} \textbf{31} (1977), 204--256.
	
	\bibitem{Fu81a} H.~Furstenberg.
	\newblock Recurrence in ergodic theory and combinatorial number theory.
	\newblock Princeton University Press, Princeton, 1981.
		
		\bibitem{GKL17} A.~Gomilko, D.~Kwietniak, M.~Lema\'nczyk.
	\newblock  Sarnak's conjecture implies the Chowla conjecture along a subsequence.
	\newblock	 (2017), 	\texttt{arXiv:1710.07049}
	 
	\bibitem{Gow01} W.~Gowers.
	\newblock  A new proof of Szemer\'edi's theorem.
	\newblock {\em Geom. Funct. Anal.} {\bf 11} (2001), 465--588.
	
	\bibitem{GS16} A.~Granville,  K.~Soundararajan.
	\newblock Multiplicative Number Theory: The pretentious approach.
	\newblock Book manuscript  in preparation.
	
	\bibitem{GT10} B.~Green, T.~Tao.
	\newblock An arithmetic regularity lemma, associated counting lemma, and applications.
	\newblock An irregular mind,  {\em Bolyai, Soc. Math. Stud.}
	{\bf 21},  J\'anos Bolyai Math. Soc., Budapest, (2010), 261--334.
	
	\bibitem{GT10b} B.~Green,  T.~Tao.
	\newblock Linear equations in the primes.
	\newblock {\em Ann. of Math.} {\bf 171} (2010), 1753--1850.
	
	\bibitem{GT12}
	B.~Green, T.~Tao.
	\newblock The quantitative behaviour of polynomial orbits on nilmanifolds.
	\newblock {\em Ann. of Math.}  {\bf 175}  (2012),  465--540.
	
	\bibitem{GT12b}
	B.~Green,   T.~Tao.
	\newblock The M\"obius function is strongly orthogonal to nilsequences.
	\newblock {\em Ann. of Math.} {\bf 175} (2012), no. 2, 541--566.
	
	\bibitem{GTZ12c} B.~Green, T.~Tao, T.~Ziegler.
	\newblock An inverse theorem for the Gowers $U^{s+1}[N]$-norm.
	\newblock {\em Ann. of Math.} {\bf 176} (2012), no. 2, 1231--1372.
	
	\bibitem{Hal68} G.~Hal\'asz.
	\newblock \"{U}ber die Mittelwerte multiplikativer zahlentheoretischer Funktionen.
	\newblock {\em  Acta Math. Acad. Sci. Hung.}  {\bf 19}  (1968), 365--403.
	
	\bibitem{HK05a} B.~Host,  B.~Kra.
	\newblock Non-conventional ergodic averages and nilmanifolds.
	\newblock 	{\em Ann. of Math.}  \textbf{161}  (2005), 397--488.
	
	\bibitem{HK09} B.~Host, B.~Kra.
	\newblock Uniformity seminorms on $l^{\infty}$ and applications.
	\newblock {\em J. Analyse Math.} \textbf{108} (2009), 219--276.
	
	\bibitem{K86} I.~K\'atai.
	\newblock A remark on a theorem of H. Daboussi.
	\newblock  {\em Acta Math. Hungar.} {\bf 47} (1986), 223--225.
	
	\bibitem{Lei05} A.~Leibman.  	
	\newblock Host-Kra and Ziegler factors and convergence of multiple averages.
	\newblock  {\em Handbook of Dynamical Systems, vol. 1B}, B. Hasselblatt and A. Katok, eds.,  Elsevier (2005), 841--853
	
	\bibitem{Lei05a} A.~Leibman.
	\newblock Pointwise convergence of ergodic averages for polynomial	sequences of rotations of a nilmanifold.
	\newblock {\em Ergodic Theory Dynam. Systems} {\bf 25}  (2005), no. 1,   201--213.	
	
	\bibitem{Les91} E.~Lesigne.
	\newblock Sur une nil-vari\'et\'e, les partiesminimales associ\'ees \`a une translation sont uniquement	ergodiques.
	\newblock {\it Ergodic Theory Dynam. Systems} \textbf{11} (1991), no. 2,   379--391.
	
	\bibitem{MR15} K.~Matom\"aki, M.~Radziwi{\l}{\l}.
	\newblock Multiplicative functions in short intervals.
	\newblock  {\em Ann. of Math.}   {\bf 183} (2016), 1015--1056.
	
	
	\bibitem{MRT15}	K.~Matom\"aki, M.~Radziwi{\l}{\l},  T.~Tao.
	\newblock An averaged form of Chowla's conjecture.
	\newblock   {\em Algebra \& Number Theory} {\bf 9} (2015), 2167--2196.
	
	\bibitem{MRT15b} K.~Matom\"aki, M.~Radziwi{\l}{\l}, T.~Tao.
	\newblock Sign patterns of the Liouville and M\"obius functions.
	\newblock  {\em Forum Math. Sigma} {\bf 4} (2016).
	
	\bibitem{MV77} H.~Montgomery, R.~Vaughan.
	\newblock  Exponential sums with multiplicative coefficients.
	\newblock {\em Invent. Math.} {\bf 43} (1977), no. 1, 69--82.
	
	\bibitem{Sa10} P.~Sarnak.
	\newblock Three lectures on the M\"obius function randomness and dynamics.
	\newblock Manuscript (2010). \texttt{http://publications.ias.edu/sarnak/paper/506}.
	
	\bibitem{Sa12} P.~Sarnak.
	\newblock  M\"obius randomness and dynamics.
	\newblock {\em  Not. S. Afr. Math. Soc.} {\bf 43} (2012), no. 2, 89--97.
	
	\bibitem{Tao12}	T.~Tao.
	\newblock Higher order Fourier analysis.
	\newblock  Graduate studies in mathematics	{\bf 142}, American Mathematical Society, (2012).
	
	\bibitem{Tao15} T.~Tao.
	\newblock The logarithmically averaged Chowla and Elliott conjectures for two-point correlations.
	\newblock  {\em Forum Math. Pi}  {\bf 4} (2016).
	
	\bibitem{T16} T.~Tao.
    \newblock Equivalence of the logarithmically averaged Chowla and Sarnak conjectures.
    \newblock In: C.~Elsholtz, P.~Grabner,  Number Theory - Diophantine Problems, Uniform Distribution and Applications.
       Springer, Cham, (2017), 391--421.
	
	\bibitem{TT17} T.~Tao, J.~Ter\"av\"ainen.
    \newblock 	The structure of logarithmically averaged correlations of multiplicative functions, with applications to the Chowla and Elliott conjectures.  	
    \newblock (2017), 	\texttt{arXiv:1708.02610v1}.
	
	
	
	\bibitem{Zi07} T.~Ziegler.
    \newblock 	Universal characteristic factors and Furstenberg averages.
    \newblock {\em J. 	Amer. Math. Soc.} \textbf{20} (2007), 53--97.
	
	
	
 		
	
	
	

\end{thebibliography}



\begin{dajauthors}
\begin{authorinfo}[Nikos]
  Nikos Frantzikinakis\\
 Department of Mathematics and Applied Mathematics\\
 University of Crete\\
   Voutes University Campus,
 Heraklion 70013, Greece\\
    frantzikinakis \imageat{} gmail\imagedot{}com  \\
  \url{http://users.math.uoc.gr/~nikosf/}
\end{authorinfo}
\end{dajauthors}

\end{document}